\documentclass[a4paper,12pt,reqno]{amsart}
\usepackage{amsmath,amsfonts,amssymb,amsthm,mathtools,autobreak,atendofenv}
\usepackage{bm,ascmac,array,framed,physics,empheq}
\usepackage{circledsteps}
\usepackage[T1]{fontenc}
\usepackage{textcomp}
\usepackage[normalem]{ulem}
\usepackage{mathrsfs,enumerate}
\usepackage[shortlabels]{enumitem}
\usepackage[all]{xy}
\usepackage[dvipdfmx]{graphicx}
\usepackage{tikz}
\usepackage{url}
\usetikzlibrary{cd,intersections,calc,arrows,matrix,shapes.geometric,positioning}
\usepackage{tcolorbox}
\usepackage{color}
\usepackage{appendix}
\usepackage{stackengine}
\usepackage{here}
\usepackage[hidelinks]{hyperref}
\hypersetup{unicode,bookmarksnumbered=true,hidelinks,final}
\theoremstyle{definition}
\newtheorem{theorem}{Theorem}[section]
\newtheorem{definition}[theorem]{Definition}
\newtheorem{proposition}[theorem]{Proposition}
\newtheorem{lemma}[theorem]{Lemma}

\newtheorem{remark}[theorem]{Remark}
\newtheorem*{proposition*}{Proposition}
\mathtoolsset{showonlyrefs=true}
\numberwithin{equation}{section}
\begin{document}
\title[Moduli of HYM connections]{The moduli space of Hermitian-Yang-Mills connections}
\author{Jun Sasaki}
\date{}
\address{Department of Mathematics, Institute of Science Tokyo, 2-12-1, O-okayama, Meguro, 152-8551, Japan}
\email{sasaki.j.ac@m.titech.ac.jp, sasaki.j.fb69@m.isct.ac.jp}
\begin{abstract}
  In this paper, we study  Hermitian-Yang-Mills connections (HYM) on a smooth Hermitian vector bundle over compact K\"{a}hler manifold. We calculate the virtual dimension of the moduli space of HYM connections and provide an analytic proof that the stability of Higgs bundles is an open condition.
\end{abstract}
\maketitle
\setcounter{section}{-1}
\section{Introduction}
Let $\left(M,g\right)$ be a compact K\"{a}hler manifold and $\left(E,h\right)$ be a smooth Hermitian vector bundle over $M$. Let $\nabla$ be a unitary connection on $E$ and $\theta$ be a smooth $\textrm{End}\left(E\right)$-valued $\left(1,0\right)$-form. We decompose $\nabla:\Omega^0\left(E\right)\to\Omega^{1,0}\left(E\right)\oplus\Omega^{0,1}\left(E\right)$ as $\nabla=\partial^E+\bar{\partial}^E$ and define an operator $D$ as $D=\nabla+\theta+\theta^*$ where $\theta^*$ is an adjoint of $\theta$ with respect of $h$. Then $D$ is called a \textit{Hermitian-Yang-Mills (HYM) connection} if it satisfies the following equations:
\begin{equation}
  R\left(\nabla\right)^{2,0}=0,\quad\bar{\partial}^{\textrm{End}\left(E\right)}\theta=0,\quad\theta\wedge\theta=0,\quad\sqrt{-1}\Lambda R\left(D\right)=c\,\textrm{id}_E \label{eq:0.2}
\end{equation}
where $\bar{\partial}^{\textrm{End}\left(E\right)}:\Omega^{1,0}\bigl(\textrm{End}\left(E\right)\bigr)\to\Omega^{1,1}\bigl(\textrm{End}\left(E\right)\bigr)$ is an operator induced by $\bar{\partial}^E$ and $c$ is also a real constant depending on $M$, $g$ and $E$. These equations are called the HYM equations.

The HYM connection was first introduced by Hitchin in $1987$ \cite{Hitchin} for a rank $2$-bundle over a compact Riemann surface of genus $g\geq 2$, and then studied on compact K\"{a}hler manifolds of general dimension. In particular, in $1988$, Simpson \cite{Simpson} established the Kobayashi-Hitchin correspondence for Higgs bundles, which asserts that a Higgs bundle over a compact K\"{a}hler manifold admits an irreducible HYM connection if and only if it is stable as Higgs bundle. Therefore HYM connections are not only subjects of differential geometry but also deeply related to algebraic geometry as HE connections are.

The aim of this paper is to calculate the virtual dimension of the moduli space of HYM connections and provide an analytic proof that the stability of Higgs bundles is an open condition. As for the dimension, Tanaka proved that on a compact K\"{a}hler surface, the HYM equations are equivalent to the Kapustin-Witten equations and Liu, Rayan, Tanaka calculated the virtual dimension of the moduli space of the solutions to the Kapustin-Witten equations \cite{Tanaka},\cite{Liu}. We also compare our results with theirs. As for the stability, Simpson proveded algebraic proof that on a projective variety, the stability of Higgs bundles is open condition \cite[Lemma 3.7]{Simpson3}. Therefore our method of proof differs from Simpson's approach.

The main result of this paper is stated as follows:
\begin{enumerate}[$\left(a\right)$]
  \item When $M$ is a surface, the virtual real dimension of the moduli space of HYM connections is given by $2+b_2-r^2\chi\left(M\right)$, where $r$ is the rank of vector bundle $E$, $b_2$ is second betti number of $M$ and $\chi\left(M\right)$ is the Euler number of $M$. In addition to the assumption concerning $M$, if the structure group o vector bundle $E$ is $SU\left(2\right)$, the virtual real dimension of moduli space is given by $-3\chi\left(M\right)$, which coincides with the virtual dimension of the moduli space of the solutions to the Kapustin-Witten equations.
  \item There is an open injection from the moduli space of irreducible HYM connections into the moduli space of simple Higgs structures.
\end{enumerate}

This paper is organized as follows.
In \textbf{Section 1}, we introduce fundamental concepts such as Higgs bundles and HYM connections.
In \textbf{Section 2}, we formalize Higgs structures and HYM connections on smooth Hermitian vector bundles. We also define the moduli spaces of Higgs structures and HYM connections.
In \textbf{Section 3}, we compute the tangent space of the moduli space of HYM connections at a given point. Furthermore, we verify that this tangent space carries a complex structure as a vector space.
In \textbf{Section 4}, we justify the computations in Section $3$ by constructing an elliptic complex. This coincides with the elliptic complex obtained by Hitchin when the base space is a compact Riemann surface \cite{Hitchin}. Additionally, we describe the relationship between the cohomology groups of this elliptic complex and those of the Dolbeault-Higgs complex, which generalizes the Dolbeault complex. These results are useful for calculating the virtual dimension of the moduli space and establishing an injection from the moduli space of irreducible HYM connections into the moduli space of simple Higgs structures in Section $6$.
In \textbf{Section 5}, we establish an analogue of the Kodaira-Serre duality theorem in the context of Higgs bundles, and use it to compute the virtual dimension of the moduli space when the base of the vector bundle is a compact K\"{a}hler surface.
In \textbf{Section 6}, we prove that there exists an open injection from the moduli space of irreducible HYM connections into the moduli space of simple Higgs structures. we refer to Bradlow's proof \cite{Bradlow} on the existence of an injection from the moduli space of stable pairs to that of simple pairs. By the Kobayashi-Hitchin correspondence for Higgs bundles, the image of this injection can be identified with the moduli space of stable Higgs structures. Finally, in the \textbf{Appendix}, we study the infinitesimal deformations of Higgs structures and the deformation complex derived from them, which play a crucial role in our main results.
\section{Preliminaries}
\subsection{Higgs Bundles}
Throughout this paper, we assume that complex manifolds are connected. Let $M$ be a complex manifold and $E$ be a smooth complex vector bundle of rank $r$ over $M$.

\begin{definition}\cite{Simpson2}
Let $D^{\prime\prime}:\Omega^0\left(E\right)\to\Omega^1\left(E\right)$ be a $\mathbb{C}$-linear mapping. A couple $\left(E,D^{\prime\prime}\right)$ is called a \textit{Higgs bundle} over $M$ if $D^{\prime\prime}$ satisfies the Leibniz rule
\begin{equation}
    D^{\prime\prime}\left(fs\right)=fD^{\prime\prime}s+s\otimes\bar{\partial}f \label{eq:1.2}
\end{equation}
for $f\in C^\infty\left(M\right),s\in\Omega^0\left(E\right)$ and the integrability condition
\begin{equation}
    D^{\prime\prime}\circ D^{\prime\prime}=0:\Omega^0\left(E\right)\to\Omega^2\left(E\right). \label{eq:1.3}
\end{equation}
When $D^{\prime\prime}$ satisfies \eqref{eq:1.2} and \eqref{eq:1.3}, $D^{\prime\prime}$ is called a \textit{Higgs structure} on $E$. Moreover a triple $\left(E,D^{\prime\prime},h\right)$ is called a \textit{Hermitian Higgs bundle} over $M$ if $\left(E,D^{\prime\prime}\right)$ is a Higgs bundle over $M$ and $h$ is a Hermitian metric on $E$.
\end{definition}

According to the decomposition $\Omega^1\left(E\right)=\Omega^{1,0}\left(E\right)\oplus\Omega^{0,1}\left(E\right)$, we write the Higgs structure on $E$ as $D^{\prime\prime}=\bar{\partial}^E+\theta$. Then by the integrability condition \eqref{eq:1.3}, $\bar{\partial}^E$ is a holomorphic structure on $E$ and $\theta$ is an $\textrm{End}\left(E\right)$-valued holomorphic $1$-form satisfying $\theta\wedge\theta=0$. The $1$-form $\theta$ is called a \textit{Higgs field}.

The Higgs structure $D^{\prime\prime}=\bar{\partial^E}+\theta$ on $E$ naturally induces that on $\textrm{End}\left(E\right)$. Its Higgs field $\theta^{\textrm{End}\left(E\right)}$ is given by $\theta^{\textrm{End}\left(E\right)}\varphi=\left[\theta,\varphi\right]$ for $\varphi\in\Omega^0\bigl(\textrm{End}\left(E\right)\bigr)$.

Let $\left(E,\bar{\partial}^E+\theta,h\right)$ be a Hermitian Higgs bundle over $M$ and $\nabla^E$ be the canonical connection on a holomorphic Hermitian vector bundle $\left(E,\bar{\partial}^E,h\right)$. We decompose $\nabla^E$ as $\nabla^E=\partial^E+\bar{\partial}^E$. We define an operator $D^\prime:\Omega^0\left(E\right)\to\Omega^1\left(E\right)$ by $D^\prime=\partial^E+\theta^*$, where $\theta^*\in\Omega^{0,1}\bigl(\textrm{End}\left(E\right)\bigr)$ is the adjoint of $\theta$ with respect to $h$. Then $D^\prime$ satisfies the integrability condition $D^{\prime}\circ D^{\prime}=0:\Omega^0\left(E\right)\to\Omega^2\left(E\right)$. Also we define an operator $D:\Omega^0\left(E\right)\to\Omega^1\left(E\right)$ as $D=D^\prime+D^{\prime\prime}$. This is a connection on $E$ and called the \textit{Hitchin-Simpson (HS) connection} on the Hermitian Higgs bundle $\left(E,\bar{\partial}^E+\theta,h\right)$ \cite[p.$4$]{Bruzzo}. By the integrability conditions for $D^\prime$ and $D^{\prime\prime}$, the curvature $R\left(D\right)$ of $D$ is given by
\begin{equation}
  R\left(D\right)=R\left(\nabla^E\right)+\left[\theta\wedge\theta^*\right]+\left(\partial^{\textrm{End}\left(E\right)}\theta+\bar{\partial}^{\textrm{End}\left(E\right)}\theta^*\right).
\end{equation}
The HS connection on a Hermitian Higgs bundle is considered as a generalization of the canonical connection on a holomorphic Hermitian vector bundle.
\subsection{Hermitian-Yang-Mills Connection}
Let $\left(M,g\right)$ be an $n$-dimensional compact Kähler manifold, $\left(E,D^{\prime\prime}=\bar{\partial}^E+\theta,h\right)$ be a Hermitian Higgs bundle of rank $r$ over $M$, $D$ be the HS connection on $\left(E,D^{\prime\prime},h\right)$ and $\omega$ be the K\"{a}hler form of $\left(M,g\right)$.
\begin{definition} \cite{Simpson2}
  Consider the following equation for the HS connection $D$ on $\left(E,D^{\prime\prime},h\right)$:
  \begin{equation}
    \sqrt{-1}\Lambda R\left(D\right)=c\,\textrm{id}_E,\quad c=\frac{2\pi\mu\left(E\right)}{\left(n-1\right)!\cdot\textrm{Vol}\left(M,g\right)} \label{eq:1.5}
  \end{equation}
  This equation is called the Hermitian-Yang-Mills (HYM) equation with HYM factor $c$ in this paper. When $D$ satisfies this equation, $D$ is called a \textit{Hermitian-Yang-Mills (HYM) connection with HYM factor $c$} on $\left(E,D^{\prime\prime}\right)$ and $h$ is called a \textit{HYM metric with HYM factor $c$} on $\left(E,D^{\prime\prime}\right)$.
\end{definition}
\begin{remark}
  Since $\left(M,g\right)$ is a compact Kähler manifold, the HYM factor $c$ is determined by $M$, $g$, and $E$, and does not depend on $h$ and $D$. Therefore when we fix a smooth complex vector bundle and a compact Kähler manifold which is its base space, we may omit the phrase ``with HYM factor $c$''.
\end{remark}

We decompose $D$ as $D=\nabla^E+\Theta$, where $\nabla^E$ is the canonical connection on a holomorphic vector bundle $\left(E,\bar{\partial}^E,h\right)$ and $\Theta=\theta+\theta^*\in\Omega^1\bigl(\textrm{Herm}\left(E,h\right)\bigr)$. Then if $\theta=0$, the HYM equation coincides with the Hermitian-Einstein equation. Thus the HYM equation can be considered as a generalization of the Hermitian-Einstein equation.
\subsection{Moduli Space}
Let $M$ be a complex manifold and $E$ a smooth complex vector bundle over $M$. In what follows, we formulate the spaces of Higgs structures and HYM connections and define their moduli spaces.

Firstly, we define the moduli space of Higgs structures. Let $\mathscr{D}^{\prime\prime}\left(E\right)$ be the set of $\mathbb{C}$-linear mappings $D^{\prime\prime}:\Omega^0\left(E\right)\to\Omega^1\left(E\right)$ satisfying the Leibniz rule \eqref{eq:1.2}. Fixing any $D^{\prime\prime}_0\in\mathscr{D}^{\prime\prime}\left(E\right)$, we have $\mathscr{D}^{\prime\prime}\left(E\right)=D^{\prime\prime}_0+\Omega^1\bigl(\textrm{End}\left(E\right)\bigr)$. Let $\mathscr{H}^{\prime\prime}\left(E\right)\subset\mathscr{D}^{\prime\prime}\left(E\right)$ be the set of $D^{\prime\prime}\in\mathscr{D}^{\prime\prime}\left(E\right)$ satisfying the integrability condition \eqref{eq:1.3}. $\mathscr{H}^{\prime\prime}\left(E\right)$ is the space of Higgs structures on $E$.

Let $\mathcal{G}\left(E\right)=GL\left(E\right)$ be the gauge transformation group of $E$. Then the right-action of $\mathcal{G}\left(E\right)$ on $\mathscr{D}^{\prime\prime}\left(E\right)$ is defined by, for $g\in\mathcal{G}\left(E\right)$ and $D^{\prime\prime}\in\mathscr{D}^{\prime\prime}\left(E\right)$,
\begin{equation}
  D^{\prime\prime}\cdot g:=g^{-1}\circ D^{\prime\prime}\circ g.
\end{equation}
This action preserves $\mathscr{H}^{\prime\prime}\left(E\right)$. $\mathscr{M}_{\textrm{Higgs}}$ denotes the quotient space $\mathscr{H}^{\prime\prime}\left(E\right)/\mathcal{G}\left(E\right)$ and is called the moduli space of Higgs structures.

Secondly, we define the moduli space of HYM connections. Let $h$ be a Hermitian metric on a smooth complex vector bundle $E$ and $\mathscr{A}\left(E\right)$ be the set of connections on $E$. If a connection $D\in\mathscr{A}\left(E\right)$ is given, then there exist naturally defined $\mathbb{C}$-linear mappings $D^\prime,D^{\prime\prime}:\Omega^0\left(E\right)\to\Omega^1\left(E\right)$ satisfying
\begin{equation}
  D = D^\prime + D^{\prime\prime} \label{eq:2.1}
\end{equation}
and the Leibniz rules. (See \cite[p.$13$]{Simpson2}.) Let $\mathscr{H}\left(E,h\right)\subset\mathscr{A}\left(E\right)$ be the set of $D\in\mathscr{A}\left(E\right)$ such that $D^{\prime\prime}$ defines the Higgs structure on $E$ in the decomposition \eqref{eq:2.1} of $D$. The space $\mathscr{H}\left(E,h\right)$ can be considered as the set of HS connections on $\left(E,h\right)$ including Higgs structures on $E$.

In what follows, let $\left(M,g\right)$ be a compact K\"{a}hler manifold. Let $\mathscr{Y}\left(E,h\right)\subset\mathscr{H}\left(E,h\right)$ be the set of $D\in\mathscr{H}\left(E,h\right)$ satisfying the HYM equation \eqref{eq:1.5}. The space $\mathscr{Y}\left(E,h\right)$ can be considered as the set of HYM connections on $\left(E,h\right)$ including Higgs structures on $E$.

Let $\mathcal{G}\left(E,h\right)=U\left(E,h\right)$ be the gauge transformation group of $\left(E,h\right)$. Then the right-action of $\mathcal{G}\left(E,h\right)$ on $\mathscr{H}\left(E,h\right)$ is defined by, for $g\in\mathcal{G}\left(E,h\right)$ and $D\in\mathscr{H}\left(E,h\right)$,
\begin{equation}
  D\cdot g:=g^{-1}\circ D\circ g.
\end{equation}
This action preserves $\mathscr{Y}\left(E,h\right)$. $\mathscr{M}_{\textrm{HYM}}$ denotes the quotient space $\mathscr{Y}\left(E,h\right)/\mathcal{G}\left(E,h\right)$ and is called the moduli space of HYM connections.

Next theorem follows from the compactness of the unitary group.
\begin{theorem}
  $\mathscr{M}_{\textrm{HYM}}$ \textit{is a Hausdorff space with respect to the} $C^\infty$\textit{-topology.}
\end{theorem}
Consider a bijection
\begin{equation}
  \mathscr{A}\left(E\right)\to\mathscr{D}^{\prime\prime}\left(E\right),\quad D^\prime+D^{\prime\prime}\mapsto D^{\prime\prime}. \label{eq:2.3}
\end{equation}
$\mathcal{G}\left(E\right)$ acts on $\mathscr{A}\left(E\right)$ from the right through this bijection, wchich is given by, for $g\in\mathcal{G}\left(E\right)$ and $D^\prime+D^{\prime\prime}\in\mathscr{A}\left(E\right)$,
\begin{equation}
  \left(D^\prime+D^{\prime\prime}\right)\cdot g=g^*\circ D^\prime\circ{g^*}^{-1}+g^{-1}\circ D^{\prime\prime}\circ g, \label{eq:2.4}
\end{equation}
where $g^*\in\mathcal{G}\left(E\right)$ is the adjoint of $g$ with respect to $h$. Similarly, there exists a bijection between $\mathscr{H}^{\prime\prime}\left(E\right)$ and $\mathscr{H}\left(E,h\right)$, which is given by the restriction of \eqref{eq:2.3} to $\mathscr{H}\left(E,h\right)$. Also, the right-action of $\mathcal{G}\left(E\right)$ on $\mathscr{A}\left(E\right)$ preserves $\mathscr{H}\left(E,h\right)$.
\section{Infinitesimal deformations of HYM connections}
In this section, we calculate the space of infinitesimal deformations of a HYM connection $D$. This coincides with the tangent space of the moduli space $\mathscr{M}_{\textrm{HYM}}$ at $\left[D\right]$. Let $\left(M,g\right)$ be an $n$-dimensional compact Kähler manifold, $\left(E,h\right)$ be a smooth Hermitian vector bundle over $M$. Let
\begin{equation}
    D_t=D+\alpha_t,\quad\left|t\right|\ll 1
\end{equation}
be a curve in $\mathscr{Y}\left(E,h\right)$, where $\alpha_t\in\Omega^1\bigl(\textrm{End}\left(E\right)\bigr)$ and $\alpha_0=0$. If the decompositions \eqref{eq:2.1} of $D_t$ and $D$ are written as $D_t=D_t^\prime+D_t^{\prime\prime}$ and $D=D^\prime+D^{\prime\prime}$, then we have
\begin{equation}
  D_t^{\prime\prime}=D^{\prime\prime}+\frac{1}{2}\left(\alpha_t+\left(\alpha_t^*\right)^\prime-\left(\alpha_t^*\right)^{\prime\prime}\right). \label{eq:3.1}
\end{equation}

Consider a $\mathbb{R}$-linear mapping $P_1:\Omega^1\bigl(\textrm{End}\left(E\right)\bigr)\to\Omega^1\bigl(\textrm{End}\left(E\right)\bigr)$ defined by, for $\alpha=\beta+\gamma\in\Omega^1\bigl(\textrm{Herm}_{\textrm{skew}}\left(E,h\right)\bigr)\oplus\Omega^1\bigl(\textrm{Herm}\left(E,h\right)\bigr)$,
\begin{equation}
  P_1\left(\alpha\right)=\beta^{\prime\prime}+\gamma^\prime \label{fig:1}
\end{equation}
where $\beta^{\prime\prime}$ and $\gamma^\prime$ is $\left(0,1\right)$ and $\left(1,0\right)$-components of $\beta$ and $\gamma$, respectively. Then $P_1$ is bijective. For $\alpha_t\in\Omega^1\bigl(\textrm{End}\left(E\right)\bigr)$, from \eqref{eq:3.1}, we have $D_t^{\prime\prime}=D^{\prime\prime}+\left(\beta_t^{\prime\prime}+\gamma_t^\prime\right)$. Since $D_t\in\mathscr{Y}\left(E,h\right)$ implies $D_t^{\prime\prime}\in\mathscr{H}^{\prime\prime}\left(E\right)$, we have
\begin{equation}
  D_{\textrm{End}\left(E\right)}^{\prime\prime}\left(\beta_t^{\prime\prime}+\gamma_t^\prime\right)+\left(\beta_t^{\prime\prime}+\gamma_t^\prime\right)\wedge\left(\beta_t^{\prime\prime}+\gamma_t^\prime\right)=0.
\end{equation}
Differentiating both sides of this equation at $t=0$, we obtain
\begin{equation}
  D_{\textrm{End}\left(E\right)}^{\prime\prime}\left(\beta^{\prime\prime}+\gamma^\prime\right)=0,\textrm{ where }\beta^{\prime\prime}=\left.\frac{d}{dt}\right|_{t=0}\beta_t^{\prime\prime},\gamma^\prime=\left.\frac{d}{dt}\right|_{t=0}\gamma_t^\prime.
\end{equation}

We decompose $D_t$ and $D$ into unitary connection part and $\Omega^1\bigl(\textrm{Herm}\left(E,h\right)\bigr)$ part as $D_t=\nabla_t^E+\Theta_t$ and $D=\nabla^E+\Theta$. Then we have $\nabla^E_t=\nabla^E+\beta_t$ and $\Theta_t=\Theta+\gamma_t$.
\if Henceforth, when we write a connection as $\nabla$ with the symbol of the vector bundle $E$ in the upper right, unless otherwise we state, it always denotes the canonical connection on a holomorphic Hermitian vector bundle $\left(E, \bar{\partial}^E, h\right)$, where $\bar{\partial}^E$ is a $\left(0,1\right)$-component of $\nabla$, that is, we assume that $\left(0,1\right)$-component of $\nabla$ is always integrable.\fi
Since $D_t$ satisfies the HYM equation $\sqrt{-1}\Lambda\Bigl(R\left(\nabla_t^E\right)+\Theta_t\wedge\Theta_t\Bigr)=0$, we have
\begin{equation}
  \sqrt{-1}\Lambda\Bigl(R\left(\nabla^E\right)+\Theta\wedge\Theta+\nabla^{\textrm{End}\left(E\right)}\beta_t+\left[\Theta,\gamma_t\right]+\beta_t\wedge\beta_t+\gamma_t\wedge\gamma_t\Bigr)=c\,\textrm{id}_E.
\end{equation}
Since $c$ is a constant independent of $t$, differentiating both sides at $t=0$, we obtain
\begin{equation}
  \Lambda\left(\nabla^{\textrm{End}\left(E\right)}\beta+\left[\Theta,\gamma\right]\right)=0,\textrm{ where }\beta=\left.\frac{d}{dt}\right|_{t=0}\beta_t,\gamma=\left.\frac{d}{dt}\right|_{t=0}\gamma_t.
\end{equation}
Moreover if $D_t$ are obtained by a one-parameter family of gauge transformations $\left(g_t\right)_{\left|t\right|\ll 1}$ of $\left(E,h\right)$ as $D_t=g_t^{-1}\circ D\circ g_t$, then differentiating both sides of this equation at $t = 0$, we obtain
\begin{equation}
  \alpha=\nabla^{\textrm{End}\left(E\right)}g+\left[\Theta,g\right],\textrm{ where }g=\left.\frac{d}{dt}\right|_{t=0}g_t.
\end{equation}
From these, the space of infinitesimal deformations of $D$ is given by $H_D^1$ defined by
\begin{equation}
  H_D^1=\frac{\left\{\beta+\gamma\in\Omega^1\bigl(\textrm{End}\left(E\right)\bigr)\middle|
  \begin{gathered}
    D_{\textrm{End}\left(E\right)}^{\prime\prime}(\beta^{\prime\prime}+\gamma^\prime)=0,\\
    \Lambda\left(\nabla^{\textrm{End}\left(E\right)}\beta+\left[\Theta,\gamma\right]\right)=0
  \end{gathered}
  \right\}}{\left\{\nabla^{\textrm{End}\left(E\right)}g+\left[\Theta,g\right]\middle|g\in\Omega^0\bigl(\textrm{Herm}_{\textrm{skew}}\left(E,h\right)\bigr)\right\}}. \label{eq:3.2}
\end{equation}
We will discuss in Section $4$ how this can be formulated as the cohomology group of a certain elliptic complex.

In the following, we will show this vector space has a complex structure. Since there is $\beta^{\prime\prime}+\gamma^\prime$ on the right side of \eqref{eq:3.2} and $\alpha\in\Omega^1\bigl(\textrm{End}\left(E\right)\bigr)$ is in one-to-one correspondence with $\beta^{\prime\prime}+\gamma^\prime\in\Omega^1\bigl(\textrm{End}\left(E\right)\bigr)$ by $P_1$ defined as in \eqref{fig:1}, we rewrite \eqref{eq:3.2} by using $\beta^{\prime\prime}$ and $\gamma^\prime$. First, we have
\begin{align}
  &\Lambda\left(\nabla^{\textrm{End}\left(E\right)}\beta+\left[\Theta,\gamma\right]\right)=0\\
  \Longleftrightarrow&\Lambda\left(\partial^{\textrm{End}\left(E\right)}\beta^{\prime\prime}+\bar{\partial}^{\textrm{End}\left(E\right)}\beta^\prime+\theta^{\textrm{End}\left(E\right)}\wedge\gamma^{\prime\prime}+\left(\theta^{\textrm{End}\left(E\right)}\right)^*\wedge\gamma^\prime\right)=0,
\end{align}
where $\Theta=\theta+\theta^*$. Next, the set $\left\{\nabla^{\textrm{End}\left(E\right)}g+\left[\Theta,g\right]\middle|g\in\Omega^0\bigl(\textrm{Herm}_{\textrm{skew}}\left(E,h\right)\bigr)\right\}$ is considered as the tangent space $T_D\bigl(D\cdot\mathcal{G}\left(E,h\right)\bigr)$ at $D$. Hence its ``orthogonal complement'' is given by $\left\{\beta+\gamma\in\Omega^1\bigl(\textrm{End}\left(E\right)\bigr)\middle|\left(d^{\nabla^{\textrm{End}\left(E\right)}}\right)^*\beta+\left(\Theta^{\textrm{End}\left(E\right)}\wedge\right)^*\gamma=0\right\}$. Note that the orthogonal complement is with respect to the real inner product which is given by, for $\alpha_1,\alpha_2\in\Omega^1\bigl(\textrm{End}\left(E\right)\bigr)$,
\begin{equation}
  \langle\alpha_1,\alpha_2\rangle:=\frac{1}{2}\int_M\left(h_{\textrm{End}\left(E\right)\otimes\Lambda^1M}\left(\alpha_1,\alpha_2\right)+h_{\textrm{End}\left(E\right)\otimes\Lambda^1M}\left(\alpha_1,\alpha_2\right)\right)\frac{\omega^n}{n!}. \label{eq:3.3}
\end{equation}
Therefore there is the following $\mathbb{R}$-linear isomorphism:
\begin{equation}
  H_D^1\simeq\left\{\beta+\gamma\in\Omega^1\bigl(\textrm{End}\left(E\right)\bigr)\middle|
  \begin{gathered}
    D_{\textrm{End}\left(E\right)}^{\prime\prime}\left(\beta^{\prime\prime}+\gamma^\prime\right)=0,\\
    \Lambda\left(\nabla^{\textrm{End}\left(E\right)}\beta+\left[\Theta,\gamma\right]\right)=0,\\
    \left(d^{\nabla^{\textrm{End}\left(E\right)}}\right)^*\beta+\left(\Theta^{\textrm{End}\left(E\right)}\wedge\right)^*\gamma=0 
  \end{gathered}\right\}
\end{equation}
By K\"{a}hler identities for Higgs bundles (see \cite[p.15]{Simpson2}), we have
\begin{align}
  &\left(d^{\nabla^{\textrm{End}\left(E\right)}}\right)^*\beta=\sqrt{-1}\Lambda\left(\bar{\partial}^{\textrm{End}\left(E\right)}\beta^\prime-\partial^{\textrm{End}\left(E\right)}\beta^{\prime\prime}\right),\\
  &\left(\Theta^{\textrm{End}\left(E\right)}\wedge\right)^*\gamma=-\sqrt{-1}\Lambda\left(\left(\theta^{\textrm{End}\left(E\right)}\right)^*\wedge\gamma^\prime-\theta^{\textrm{End}\left(E\right)}\wedge\gamma^{\prime\prime}\right).
\end{align}
Thus we have
\begin{align}
  &\left\{
  \begin{aligned}
    &\Lambda\left(\nabla^{\textrm{End}\left(E\right)}\beta+\left[\Theta,\gamma\right]\right)=0,\\
    &\left(d^{\nabla^{\textrm{End}\left(E\right)}}\right)^*\beta+\left(\Theta^{\textrm{End}\left(E\right)}\wedge\right)^*\gamma=0\\
  \end{aligned}\right.\\
  \Longleftrightarrow&\left(\bar{\partial}^{\textrm{End}\left(E\right)}\right)^*\beta^{\prime\prime}+\left(\theta^{\textrm{End}\left(E\right)}\wedge\right)^*\gamma^\prime=0.
\end{align}
Combining these, we obtain
\begin{align}
  H_D^1&\simeq\left\{\beta^{\prime\prime}+\gamma^\prime\in\Omega^1\bigl(\textrm{End}\left(E\right)\bigr)\middle|
  \begin{gathered}
    D_{\textrm{End}\left(E\right)}^{\prime\prime}\left(\beta^{\prime\prime}+\gamma^\prime\right)=0\\
    \left(\bar{\partial}^{\textrm{End}\left(E\right)}\right)^*\beta^{\prime\prime}+\left(\theta^{\textrm{End}\left(E\right)}\wedge\right)^*\gamma^\prime=0
  \end{gathered}\right\}\\
  &\simeq\left\{\alpha\in\Omega^1\bigl(\textrm{End}\left(E\right)\bigr)\middle|
  \begin{gathered}
    D_{\textrm{End}\left(E\right)}^{\prime\prime}\alpha=0,\\
    \left(D_{\textrm{End}\left(E\right)}^{\prime\prime}\right)^*\alpha=0
  \end{gathered}\right\}. \label{eq:3.4}
\end{align}
The vector space \eqref{eq:3.4} is the space of infinitesimal deformations of a Higgs structure $D^{\prime\prime}$ (cf. Appendix). Thus since the space \eqref{eq:3.4} has a natural complex structure, we obtain that the tangent space of $\mathscr{M}_{\textrm{HYM}}$ also has a complex structure through the above isomorphism.
\section{Deformation Complex of HYM Connections}
In this section, we construct a deformation complex based on a discussion in section $3$. Let $\left(M,g\right)$ be an $n$-dimensional compact Kähler manifold, $\omega$ be the K\"{a}hler form and $\left(E,h\right)$ be a smooth Hermitian vector bundle of rank $r$ over $M$. We define for $k,p,q\in\mathbb{Z}_{\ge 0}$,
\begin{gather}
  \mathscr{B}^k:=\Omega^k\bigl(\textrm{Herm}_{\textrm{skew}}\left(E,h\right)\bigr),\quad\mathscr{C}^k:=\Omega^k\bigl(\textrm{End}\left(E\right)\bigr),\\
  \mathscr{B}^{p,q}:=\Omega^{p,q}\bigl(\textrm{Herm}_{\textrm{skew}}\left(E,h\right)\bigr)=\Omega^0\bigl(\textrm{Herm}_{\textrm{skew}}\left(E,h\right)\otimes_\mathbb{R}\Lambda^{p,q}M\bigr)
\end{gather}
and define
\begin{align}
  &\mathscr{B}^2_+:=\mathscr{B}^2\cap\left(\mathscr{B}^{2,0}\oplus\mathscr{B}^{0,2}\oplus\mathscr{B}^0\omega\right),\\
  &\mathscr{B}^2_-:=\left\{\beta^{1,1}\in\mathscr{B}^{1,1}\middle|\beta^{1,1}=\overline{\beta^{1,1}},\Lambda\beta^{1,1}=0\right\}.
\end{align}
Since $\mathscr{B}^2=\mathscr{B}^2_+\oplus\mathscr{B}^2_-$ holds, we define an operator $P_0$ as the projection from $\mathscr{B}^2$ to $\mathscr{B}^0$.

For $D=\nabla^E+\Theta\in\mathscr{Y}\left(E,h\right)$, consider the following sequence:
\begin{equation}
  \left(\mathscr{B}^*\right):\xymatrix{0\ar[r]&\mathscr{B}^0\ar[r]^-{D_{\textrm{End}\left(E\right)}}&\mathscr{C}^1\ar[r]^-{D_0}&\mathscr{B}^0\oplus\mathscr{C}^2\ar[r]^-{D_2}&\mathscr{C}^3\ar[r]^-{D_{\textrm{End}\left(E\right)}^{\prime\prime}}&\cdots\ar[r]^-{D_{\textrm{End}\left(E\right)}^{\prime\prime}}&\mathscr{C}^{2n}\ar[r]&0}
\end{equation}
where $D_0:\mathscr{C}^1\to\mathscr{B}^0\oplus\mathscr{C}^2$ is defined by, for $\alpha_1=\beta_1+\gamma_1\in\Omega^1\bigl(\textrm{Herm}_{\textrm{skew}}\left(E,h\right)\bigr)\oplus\Omega^1\bigl(\textrm{Herm}\left(E,h\right)\bigr)$,
\begin{equation}
  D_0\left(\alpha_1\right)=\biggl(P_0\left(d^{\nabla^{\textrm{End}\left(E\right)}}\beta_1+\Theta^{\textrm{End}\left(E\right)}\wedge\gamma_1\right),D_{\text{End}\left(E\right)}^{\prime\prime}\left(\beta_1^{\prime\prime}+\gamma_1^\prime\right)\biggr)
\end{equation}
and $D_2:\mathscr{B}^0\oplus\mathscr{C}^2\to\mathscr{C}^3$ is defined by, for $\left(\varphi,\alpha_2\right)\in\mathscr{B}^0\oplus\mathscr{C}^2,D_2\left(\varphi,\alpha_2\right)=D_{\textrm{End}\left(E\right)}^{\prime\prime}\left(\alpha_2\right)$. By a straightforward calculation, we can prove $\left(\mathscr{B}^*\right)$ is a complex. Therefore we will show that $\left(\mathscr{B}^*\right)$ is an elliptic complex.
\begin{theorem}
  \textit{The sequence} $\left(\mathscr{B}^*\right)$ \textit{is an elliptic complex.}
\end{theorem}
\begin{proof}
  We calculate the symbols of $D_{\textrm{End}\left(E\right)}:\mathscr{B}^0\to\mathscr{C}^1,D_0:\mathscr{B}^1\to\mathscr{B}^0\oplus\mathscr{C}^2$ and $D_2:\mathscr{B}^0\oplus\mathscr{C}^2\to\mathscr{C}^3$. Since these are differential operators of order 1, for $p\in M,\xi\in T_p^*M\setminus\left\{0\right\}$, we have
  \begin{align}
    \left\{
      \begin{aligned}
        &\sigma_1\left(D_{\textrm{End}\left(E\right)}\right)\left(\xi\right)\left(\varphi\right)=\xi\otimes\varphi,&&\varphi\in\mathscr{B}_p^0\\
        &\sigma_1\left(D_0\right)\left(\xi\right)\left(\alpha_1\right)=\bigl(P_0\left(\xi\wedge\beta_1\right),\xi^{\prime\prime}\wedge\left(\beta_1^{\prime\prime}+\gamma_1^\prime \right)\bigr),&&\alpha_1\in\mathscr{C}_p^1\\
        &\sigma_1\left(D_2\right)\left(\xi\right)\left(\varphi,\alpha_2\right)=\xi^{\prime\prime}\wedge\alpha_2,&&\varphi\in\mathscr{B}_p^0,\alpha_2\in\mathscr{C}_p^2.
      \end{aligned}
    \right.
  \end{align}
  Since
  \begin{equation}
    \xymatrix{0\ar[r]&\mathscr{B}_p^0\ar[rr]^-{\sigma_1\left(D_{\textrm{End}\left(E\right)}\right)\left(\xi\right)}&&\mathscr{C}_p^1}
  \end{equation}
  and
  \begin{equation}
    \xymatrix{\mathscr{C}_p^3\ar[rr]^-{\sigma_1\left(D_{\textrm{End}\left(E\right)}^{\prime\prime}\right)\left(\xi\right)}&&\mathscr{C}_p^4\ar[rr]^-{\sigma_1\left(D_{\textrm{End}\left(E\right)}^{\prime\prime}\right)\left(\xi\right)}&&\cdots\ar[rr]^-{\sigma_1\left(D_{\textrm{End}\left(E\right)}^{\prime\prime}\right)\left(\xi\right)}&&\mathscr{C}_p^{2n}\ar[r]&0}
  \end{equation}
  are exact, we will show that
  \begin{equation}
    \xymatrix{\mathscr{B}_p^0\ar[rr]^-{\sigma_1\left(D_{\textrm{End}\left(E\right)}\right)\left(\xi\right)}&&\mathscr{C}_p^1\ar[rr]^-{\sigma_1\left(D_0\right)\left(\xi\right)}&&\mathscr{B}_p^0\oplus\mathscr{C}_p^2\ar[rr]^-{\sigma_1\left(D_2\right)\left(\xi\right)}&&\mathscr{C}_p^3\ar[rr]^-{\sigma_1\left(D_{\textrm{End}\left(E\right)}^{\prime\prime}\right)\left(\xi\right)}&&\mathscr{C}_p^4}
  \end{equation}
  is exact.
  
  First, let $\alpha\in\mathscr{C}_p^1$ satisfy that
  \begin{equation}
    \sigma_1\left(D_0\right)\left(\xi\right)\left(\alpha\right)=0 \label{eq:4.2}
  \end{equation}
  and we decompose it as
  \begin{gather}
    \alpha=\beta+\gamma\in\bigl(\textrm{Herm}_{\textrm{skew}}\left(E,h\right)\otimes\Lambda^1M\bigr)_p\oplus\bigl(\textrm{Herm}\left(E,h\right)\otimes\Lambda^1M\bigr)_p
  \end{gather}
  From \eqref{eq:4.2}, we have $\xi^{\prime\prime}\wedge\left(\beta^{\prime\prime}+\gamma^{\prime}\right)=0$. Thus we obtain $\xi^{\prime\prime}\wedge\beta^{\prime\prime}=0$ and $\xi^{\prime\prime}\wedge\gamma^\prime=0$. Therefore there exists a $\varphi\in\mathscr{B}_p^0$ such that $\beta^{\prime\prime}=\xi^{\prime\prime}\otimes\varphi$ and $\gamma^\prime=0$. Since $\gamma^{\prime\prime}=\left(\gamma^\prime\right)^*=0$, we conclude that $\alpha=\beta$ and
  \begin{equation}
    \alpha=\beta=-\left(\beta^{\prime\prime}\right)^*+\beta^{\prime\prime}=\left(\xi^\prime+\xi^{\prime\prime} \right)\otimes\varphi=\sigma_1\left(D_{\textrm{End}\left(E\right)}\right)\left(\xi\right)\left(\varphi\right).
  \end{equation}
  Therefore the sequence
  \begin{equation}
    \xymatrix{\mathscr{B}_p^0\ar[rr]^-{\sigma_1\left(D_{\textrm{End}\left(E\right)}\right)\left(\xi\right)}&&\mathscr{C}_p^1\ar[rr]^-{\sigma_1\left(D_0\right)\left(\xi\right)}&&\mathscr{B}_p^0\oplus\mathscr{C}_p^2}
  \end{equation}
  is exact.
  
  Next, for $\alpha\in\mathscr{C}_p^3$ satisfying $\sigma_1\left(D_{\textrm{End}\left(E\right)}^{\prime\prime}\right)\left(\xi\right)\left(\alpha\right)=0$, there exists $\beta\in\mathscr{C}_p^2$ such that $\alpha=\xi^{\prime\prime}\wedge\beta$. Thus we can express $\alpha$ as $\alpha=\sigma_1\left(D_2\right)\left(\xi\right)\left(0,\beta\right)$ and the sequence
  \begin{equation}
    \xymatrix{\mathscr{B}_p^0\oplus\mathscr{C}_p^2\ar[rr]^-{\sigma_1\left(D_2\right)\left(\xi\right)}&&\mathscr{C}_p^3\ar[rr]^-{\sigma_1\left(D_{\textrm{End}\left(E\right)}^{\prime\prime}\right)\left(\xi\right)}&&\mathscr{C}_p^4}
  \end{equation}
  is exact.
  
  Finally, since we have
  \begin{align}
    &\dim\mathscr{B}_p^0-\dim\mathscr{C}_p^1+\dim\mathscr{B}_p^0+\dim\mathscr{C}_p^2-\dim\mathscr{C}_p^3+\cdots+\left(-1\right)^{2n}\dim\mathscr{C}_p^{2n}\\
    =&r^2\left(1-2\cdot\dbinom{2n}{1}+1+2\cdot\dbinom{2n}{2}+2\sum_{i=3}^{2n}\left(-1\right)^i\dbinom{2n}{i}\right)=0,
  \end{align}
  the sequence
  \begin{equation}
    \xymatrix{\mathscr{C}_p^1\ar[rr]^-{\sigma_1\left(D_0\right)\left(\xi\right)}&&\mathscr{B}_p^0\oplus\mathscr{C}_p^2\ar[rr]^-{\sigma_1\left(D_2\right)\left(\xi\right)}&&\mathscr{C}_p^3}
  \end{equation}
  is exact. From these, $\left(\mathscr{B}^*\right)$ is an elliptic complex.
\end{proof}
The elliptic complex $\left(\mathscr{B}^*\right)$ is considered as a generalization of an elliptic complex obtained by Kim \cite[p.$226$]{Kobayashi}. \if In fact, for $D=\nabla+\Theta\in\mathscr{Y}\left(E,h\right)$, if $\Theta=0$, $\left(\mathscr{B}^*\right)$ is written as
\begin{figure}[H]
  \centering
  \begin{tikzpicture}
    \matrix[matrix of math nodes, row sep=0.5cm, column sep=1cm]
    {&&&|(DBDBD)|\left(\bar{\partial}^{\textrm{End}\left(E\right)}\beta_1^\prime+\partial^{\textrm{End}\left(E\right)}\beta_1^{\prime\prime},\bar{\partial}^{\textrm{End}\left(E\right)}\beta_1^{\prime\prime}\right)\\
    &&|(BB)|\beta_1+\sqrt{-1}\beta_2&|(DBDB)|\left(P_0\left(\nabla^{\textrm{End}\left(E\right)}\beta_1\right),\bar{\partial}^{\textrm{End}\left(E\right)}\beta_1^{\prime\prime}\right)\\
    |(ZERO)|0&|(B)|\mathscr{B}^0&|(E)|\mathscr{B}^1\oplus\sqrt{-1}\mathscr{B}^1&|(BC)|\mathscr{B}\oplus\mathscr{C}^2&|(CCC)|\mathscr{C}^3&|(DOT)|\cdots\\
    &|(V)|\varphi&|(DV)|\nabla^{\textrm{End}\left(E\right)}\varphi+0&|(BG)|\left(\varphi^\prime,\gamma\right)&|(DG)|\bar{\partial}^{\textrm{End}\left(E\right)}\gamma\\};
    \path[-stealth]
    (DBDBD)edge[draw=none]node[align=center,rotate=90]{$=$}(DBDB)
    (BB)edge[|->](DBDB)
    (E)edge[draw=none]node[align=center,rotate=-90]{$\in$}(BB)
    (BC)edge[draw=none]node[align=center,rotate=-90]{$\in$}(DBDB)
    (ZERO)edge[->](B)
    (B)edge[->](E)
    (E)edge[->](BC)
    (BC)edge[->](CCC)
    (CCC)edge[->](DOT)
    (B)edge[draw=none]node[align=center,rotate=90]{$\in$}(V)
    (E)edge[draw=none]node[align=center,rotate=90]{$\in$}(DV)
    (BC)edge[draw=none]node[align=center,rotate=90]{$\in$}(BG)
    (CCC)edge[draw=none]node[align=center,rotate=90]{$\in$}(DG)
    (V)edge[|->](DV)
    (BG)edge[|->](DG);
  \end{tikzpicture}
\end{figure}
where $\nabla^{\textrm{End}\left(E\right)}=\partial^{\textrm{End}\left(E\right)}+\bar{\partial}^{\textrm{End}\left(E\right)}$.\fi

Let $H_D^k\left(\mathscr{B}^*\right)$ be the $k$-th cohomology group of $\left(\mathscr{B}^*\right)$. Then the first cohomology group $H^1_D\left(\mathscr{B}^*\right)$ coincides with the space of infinitesimal deformations $H_D^1$.

Next, we describe the relationships between the cohomology groups of the elliptic complex $\left(\mathscr{C}^*\right)$ defined in Appendix and those of $\left(\mathscr{B}^*\right)$, which will be used in Section $5$ and $6$. To describe the relationships, we use the following lemma.
\begin{lemma}\ 

  \begin{enumerate}[$\left(1\right)$]
    \item \textit{For} $\beta_1,\beta_2\in\Omega^1\bigl(\textrm{Herm}_{\textrm{skew}}\left(E,h\right)\bigr)$, \textit{we have}
    \begin{equation}
      g_{\textrm{Herm}_{\textrm{skew}}\left(E,h\right)\otimes\Lambda^1M}\left(\beta_1,\beta_2\right)=2g_{\textrm{End}\left(E\right)\otimes\Lambda^{0,1}M}\left(\beta_1^{\prime\prime},\beta_2^{\prime\prime}\right).
    \end{equation}
    \item \textit{For} $\gamma_1,\gamma_2\in\Omega^1\bigl(\textrm{Herm}\left(E,h\right)\bigr)$, \textit{we have}
    \begin{equation}
      g_{\textrm{Herm}\left(E,h\right)\otimes\Lambda^1M}\left(\gamma_1,\gamma_2\right)=2g_{\textrm{End}\left(E\right)\otimes\Lambda^{1,0}M}\left(\gamma_1^\prime,\gamma_2^\prime\right).
    \end{equation}
    \item \textit{With respect to the real inner product} $\langle\cdot,\cdot\rangle$ \textit{on} $\mathscr{C}^1$ \textit{defined by} \eqref{eq:3.3}, \textit{the adjoint} $P_1^*$ \textit{of the bijection} $P_1$ \textit{defined as in} \eqref{fig:1} \textit{coincides with} $\displaystyle\frac{1}{2}P_1^{-1}$.
  \end{enumerate}
\end{lemma}
Since this lemma can be shown by a direct calculation, we omit a proof.
\begin{theorem}\label{thm:4.4}
  \textit{For} $D\in\mathscr{Y}\left(E,h\right)$, \textit{the following properties hold between the cohomology groups of} $\left(\mathscr{B}^*\right)$ \textit{and those of} $\left(\mathscr{C}^*\right)$.
  \begin{enumerate}[$\left(1\right)$]
    \item $H_D^0\left(\mathscr{B}^*\right)\otimes\mathbb{C}\simeq H_{D^{\prime\prime}}^0\left(\mathscr{C}^*\right)$.
    \item $H_D^1\left(\mathscr{B}^*\right)\simeq H_{D^{\prime\prime}}^1\left(\mathscr{C}^*\right)$.
    \item $H_D^2\left(\mathscr{B}^*\right)\simeq H_{D^{\prime\prime}}^2\left(\mathscr{C}^*\right)\oplus H_D^0\left(\mathscr{B}^*\right)$.
    \item $H_D^q\left(\mathscr{B}^*\right)\simeq H_{D^{\prime\prime}}^q\left(\mathscr{C}^*\right),\,3\leq q$.
  \end{enumerate}
\end{theorem}
\begin{proof}\ 

  \begin{enumerate}[$\left(1\right)$, leftmargin=20pt]
    \item We decompose $D_{\textrm{End}\left(E\right)}$ into unitary connection part and Hermitian transformation of $\textrm{End}\left(E\right)$-valued $1$-form part as $D_{\textrm{End}\left(E\right)}=\nabla^{\textrm{End}\left(E\right)}+\Theta^{\textrm{End}\left(E\right)}$. Then for $\varphi_1,\varphi_2\in H_D^0\left(\mathscr{B}^*\right)$, we have
    \begin{equation}
      \nabla^{\textrm{End}\left(E\right)}\varphi_1=\nabla^{\textrm{End}\left(E\right)}\varphi_2=0\in\Omega^1\bigl(\textrm{Herm}_{\textrm{skew}}\left(E,h\right)\bigr) \label{eq:4.5}
    \end{equation}
    \begin{equation}
      \Theta^{\textrm{End}\left(E\right)}\varphi_1=\Theta^{\textrm{End}\left(E\right)}\varphi_2=0\in\Omega^1\bigl(\textrm{Herm}\left(E,h\right)\bigr) \label{eq:4.6}
    \end{equation}
    From \eqref{eq:4.5}, we have
    \begin{equation}
      \bar{\partial}^{\textrm{End}\left(E\right)}\varphi_1=\bar{\partial}^{\textrm{End}\left(E\right)} \varphi_2=0
    \end{equation}
    and from \eqref{eq:4.6}, we have
    \begin{equation}
      \theta^{\textrm{End}\left(E\right)}\wedge\varphi_1=\theta^{\textrm{End}\left(E\right)}\wedge\varphi_2=0.
    \end{equation}
    Therefore we obtain
    \begin{equation}
      D_{\textrm{End}\left(E\right)}^{\prime\prime}\varphi_1=D_{\textrm{End}\left(E\right)}^{\prime\prime}\varphi_2=0.
    \end{equation}
    Thus $\varphi_1+\sqrt{-1}\varphi_2\in\mathscr{C}^0$ gives $\varphi_1+\sqrt{-1}\varphi_2\in H_{D^{\prime\prime}}^0\left(\mathscr{C}^*\right)$.

    Conversely, for $\varphi\in H_{D^{\prime\prime}}^0\left(\mathscr{C}^*\right)$, we decompose it as $\varphi=\varphi_1+\sqrt{-1}\varphi_2$ for some $\varphi_1$ and $\varphi_2\in\mathscr{B}^0$. Then since $\varphi$ satisfies $D_{\textrm{End}\left(E\right)}^{\prime\prime}\varphi=0$, we have $D_{\textrm{End}\left(E\right)}^\prime\varphi^*=0$. Also, since $D\in\mathscr{Y}\left(E,h\right)$, we have
    \begin{equation}
      \sqrt{-1}\Lambda D_{\textrm{End}\left(E\right)}^{\prime\prime}\left(D_{\textrm{End}\left(E\right)}^\prime\varphi\right)=0,\quad\sqrt{-1}\Lambda D_{\textrm{End}\left(E\right)}^\prime\left(D_{\textrm{End}\left(E\right)}^{\prime\prime}\varphi^*\right)=0.
    \end{equation}
    From these and K\"{a}hler identities for Higgs bundles, we have
    \begin{equation}
      \left(D_{\textrm{End}\left(E\right)}^\prime\right)^*D_{\textrm{End}\left(E\right)}^\prime\varphi=0,\quad\left(D_{\textrm{End}\left(E\right)}^{\prime\prime}\right)^*D_{\textrm{End}\left(E\right)}^{\prime\prime}\varphi^*=0.
    \end{equation}
    So $\varphi$ and $\varphi^*$ satisfy $D_{\textrm{End}\left(E\right)}^\prime\varphi=0$ and $D_{\textrm{End}\left(E\right)}^{\prime\prime}\varphi^*=0$, respectively. Therefore we obtain $D_{\textrm{End}\left(E\right)}\varphi=0$ and $D_{\textrm{End}\left(E\right)}\varphi^*=0$. Thus with
    \begin{equation}
      \varphi_1=\frac{1}{2}\left(\varphi-\varphi^*\right),\quad\varphi_2=\frac{1}{2}\left(\varphi+\varphi^*\right),
    \end{equation}
    we have $\varphi_1,\varphi_2\in H_D^0\left(\mathscr{B}^*\right)$.
    \item In Section $3$, we discussed that $P_1:\Omega^1\bigl(\textrm{End}\left(E\right)\bigr)\to\Omega^1\bigl(\textrm{End}\left(E\right)\bigr)$ induces a linear isomorphism between $H_D^1\left(\mathscr{B}^*\right)$ and $H_{D^{\prime\prime}}^1\left(\mathscr{C}^*\right)$, thus the result holds.
    \item First, recall that $D_0$ is defined by, for $\alpha=\beta+\gamma\in\mathscr{C}^1$,
    \begin{equation}
      D_0\left(\alpha\right)=\biggl(P_0\left(d^{\nabla^{\textrm{End}\left(E\right)}}\beta+\Theta^{\textrm{End}\left(E\right)}\gamma\right),D_{\textrm{End}\left(E\right)}^{\prime\prime}\bigl(P_1\left(\alpha\right)\bigr)\biggr).
    \end{equation}
    Therefore, by a straightforward calculation, we obtain
    \begin{equation}
      D_0^*\left(\varphi+\alpha_2\right)=\frac{1}{n}\left(d^{D_{\textrm{End}\left(E\right)}}\right)^*\left(\varphi\omega\right)+\frac{1}{2}P_1^{-1}\left(D_{\textrm{End}\left(E\right)}^{\prime\prime}\right)^*\alpha_2.
    \end{equation}
    Note that in the above calculation,
    \begin{align}
      &\left(d^{\nabla^{\textrm{End}\left(E\right)}}\right)^*\left(\varphi\omega\right)\in\Omega^1\bigl(\textrm{Herm}_{\textrm{skew}}\left(E,h\right)\bigr),\\
      &\left(\Theta^{\textrm{End}\left(E\right)}\wedge\right)^*\left(\varphi\omega\right)\in\Omega^1\bigl(\textrm{Herm}\left(E,h\right)\bigr)
    \end{align}
    hold. Next, we have
    \begin{equation}
      \left(d^{D_{\textrm{End}\left(E\right)}}\right)^*\left(\varphi\omega\right)=\left(D_{\textrm{End}\left(E\right)}^\prime\right)^*\left(\varphi\omega\right)+\left(D_{\textrm{End}\left(E\right)}^{\prime\prime}\right)^*\left(\varphi\omega\right)
    \end{equation}
    and
    \begin{align}
      \left(D_{\textrm{End}\left(E\right)}^\prime\right)^*\left(\varphi\omega\right)=-\sqrt{-1}D_{\textrm{End}\left(E\right)}^{\prime\prime}\varphi=-\sqrt{-1}\bar{\partial}^{\textrm{End}\left(E\right)}\varphi-\sqrt{-1}\left[\theta,\varphi\right].
    \end{align}
    Similarly, we have
    \begin{equation}
      \left(D_{\textrm{End}\left(E\right)}^{\prime\prime}\right)^*\left(\varphi\omega\right)=\sqrt{-1}D_{\textrm{End}\left(E\right)}^\prime\varphi=\sqrt{-1}\partial^{\textrm{End}\left(E\right)}\varphi+\sqrt{-1}\left[\theta^*,\varphi\right].
    \end{equation}
    Define
    \begin{align}
      &A:=-\sqrt{-1}\bar{\partial}^{\textrm{End}\left(E\right)}\varphi\in\Omega^{0,1}\bigl(\textrm{End}\left(E\right)\bigr),\\
      &B:=-\sqrt{-1}\left[\theta,\varphi\right]\in\Omega^{1,0}\bigl(\textrm{End}\left(E\right)\bigr).
    \end{align}
    Then we have $\left(d^{D_{\textrm{End}\left(E\right)}}\right)^*\left(\varphi\omega\right)=\left(A-A^*\right)+\left(B+B^*\right)$. So we obtain
    \begin{equation}
      P_1\Bigl(\left(d^{D_{\textrm{End}\left(E\right)}}\right)^*\left(\varphi\omega\right)\Bigr)=A+B=-\sqrt{-1}D_{\textrm{End}\left(E\right)}^{\prime\prime}\varphi.
    \end{equation}
    Therefore we have
    \begin{equation}
      D_0^*\left(\varphi+\alpha_2\right)=0\Longleftrightarrow-\frac{\sqrt{-1}}{n}D_{\textrm{End}\left(E\right)}^{\prime\prime}\varphi+\frac{1}{2}\left(D_{\textrm{End}\left(E\right)}^{\prime\prime}\right)^*\alpha_2=0. \label{eq:4.7}
    \end{equation}
    Applying $\left(D_{\textrm{End}\left(E\right)}^{\prime\prime}\right)^*$ to both sides of \eqref{eq:4.7}, we have $\left(D_{\textrm{End}\left(E\right)}^{\prime\prime}\right)^*D_{\textrm{End}\left(E\right)}^{\prime\prime}\varphi=0$, which yields
    \begin{equation}
      D_{\textrm{End}\left(E\right)}^{\prime\prime}\varphi=0. \label{eq:4.8}
    \end{equation}
    From \eqref{eq:4.7} and \eqref{eq:4.8}, we also obtain
    \begin{equation}
      \left(D_{\textrm{End}\left(E\right)}^{\prime\prime}\right)^*\alpha_2=0. \label{eq:4.9}
    \end{equation}

    Conversely, if $\varphi$ and $\alpha_2$ satisfy \eqref{eq:4.8} and \eqref{eq:4.9}, then following the argument backwards from \eqref{eq:4.7}, we have $D_0^*\left(\varphi+\alpha_2\right)=0$. Therefore we obtain
    \begin{align}
      D_0^*\left(\varphi+\alpha_2\right)=0\Longleftrightarrow\left\{
      \begin{aligned}
        &D_{\textrm{End}\left(E\right)}^{\prime\prime}\varphi=0,\\
        &\left(D_{\textrm{End}\left(E\right)}^{\prime\prime}\right)^*\alpha_2=0.
      \end{aligned}\right.
    \end{align}
    Note that from the discussion in $\left(1\right)$, for $\varphi\in\mathscr{B}^0$,
    \begin{equation}
      D_{\textrm{End}\left(E\right)}^{\prime\prime}\varphi=0\Longleftrightarrow D_{\textrm{End}\left(E\right)}\varphi=0
    \end{equation}
    holds and this proves the statement.
    \item This is obvious.
  \end{enumerate}
\end{proof}
Let $\widetilde{\mathscr{B}}^k:=\Omega^k\bigl(\textrm{Herm}_{\textrm{skew}}^0\left(E,h\right)\bigr)$ and $\widetilde{\mathscr{C}}^k:=\Omega^k\bigl(\textrm{End}^0\left(E\right)\bigr)$ and define a subcomplex $(\widetilde{\mathscr{B}}^*)$ of $\left(\mathscr{B}^*\right)$. Then $(\widetilde{\mathscr{B}}^*)$ is also an elliptic complex and if $\widetilde{H}_D^k\left(\mathscr{B}^*\right)$ denote the $k$-th cohomology group of $(\widetilde{\mathscr{B}}^*)$, similar relationships hold for the cohomology groups of $(\widetilde{\mathscr{C}}^*)$ (cf. Appendix) and those of $(\widetilde{\mathscr{B}}^*)$, that is, we have
\begin{align}
  &\widetilde{H}_D^0\left(\mathscr{B}^*\right)\otimes\mathbb{C}\simeq\widetilde{H}_{D^{\prime\prime}}^0\left(\mathscr{C}^*\right), \label{eq:4.10}\\
  &\widetilde{H}_D^1\left(\mathscr{B}^*\right)\simeq\widetilde{H}_{D^{\prime\prime}}^1\left(\mathscr{C}^*\right), \label{eq:4.11}\\
  &\widetilde{H}_D^2\left(\mathscr{B}^*\right)\simeq\widetilde{H}_{D^{\prime\prime}}^2\left(\mathscr{C}^*\right)\oplus\widetilde{H}_D^0\left(\mathscr{B}^*\right), \label{eq:4.12}\\
  &\widetilde{H}_D^q\left(\mathscr{B}^*\right)\simeq\widetilde{H}_{D^{\prime\prime}}^q\left(\mathscr{C}^*\right),\,3\leq q.
\end{align}
\section{dimension of the moduli space of HYM connections}
Let $\left(M,g\right)$ be an $n$-dimensional compact Kähler manifold and $\left(E,h\right)$ be a smooth Hermitian vector bundle of rank $r$ over $M$. Using a standard method in \cite{Griffiths}, we obtain the following theorem.
\begin{theorem}
  \textit{The moduli space of irreducible HYM connections, which is denoted by $\widetilde{\mathscr{M}}_{\textrm{HYM}}$ is a complex analytic space. Moreover, it is non-singular at a point $\left[D\right]\in\widetilde{\mathscr{M}}_{\textrm{HYM}}$ satisfying $\widetilde{H}_D^2\left(\mathscr{B}^*\right)=0$.}
\end{theorem}
\begin{remark}
  For a HYM connection $D$ on $\left(E,h\right)$, by the definition of the zeroth cohomology group $H_D^0\left(\mathscr{B}^*\right)$, $D$ is irreducible if and only if $\widetilde{H}_D^0\left(\mathscr{B}^*\right)$ vanishes.
\end{remark}
In this section, we compute the dimension of the smooth locus of the moduli space of irreducible HYM connections $\widetilde{\mathscr{M}}_{\textrm{HYM}}$. Let $\left[D\right]\in\widetilde{\mathscr{M}}_{\textrm{HYM}}$ be a smooth point of the moduli space. From the above theorem and remark, $\widetilde{H}_D^0\left(\mathscr{B}^*\right)$ and $\widetilde{H}_D^2\left(\mathscr{B}^*\right)$ vanish. Therefore from \eqref{eq:4.10},\eqref{eq:4.12} and \eqref{eq:ap} in Appendix, we have
\begin{equation}
  H_{D^{\prime\prime}}^0\left(\mathscr{C}^*\right)=\mathbb{C}\,\textrm{id}_E,\quad H_{D^{\prime\prime}}^2\left(\mathscr{C}^*\right)\simeq H^2\left(M,\mathbb{C}\right). \label{eq:4.1}
\end{equation}
Moreover, from \eqref{eq:4.11}, the dimension of $\widetilde{\mathscr{M}}_{\textrm{HYM}}$ equals that of $H_{D^{\prime\prime}}^1\left(\mathscr{C}^*\right)$. So we compute the dimension of $H_{D^{\prime\prime}}^1\left(\mathscr{C}^*\right)$.

Before computing, we derive a useul identity for the dimensions of the cohomology groups $H_{D^{\prime\prime}}^k\left(\mathscr{C}^*\right)$.
\begin{lemma}\label{Lem:4.3}
  \textit{Let $\left(E,D^{\prime\prime},h\right)$ be a Hermitian Higgs bundle over an $n$-dimensional compact K\"{a}hler manifold $\left(M,g\right)$. $\left(E^*,D_{E^*}^{\prime\prime},h_{E^*}\right)$ denotes the dual bundle of $\left(E,D^{\prime\prime},h\right)$. For $k\in\mathbb{N}$, we define $H^k\left(E\right)$ as}
  \[H_{D^{\prime\prime}}^k\left(E\right):=\left\{\alpha\in\Omega^k\left(E\right)\middle|D^{\prime\prime}\alpha=0,\left(D^{\prime\prime}\right)^*\alpha=0\right\}.\]
  \textit{Then there is an anti-linear isomorphism between $H_{D^{\prime\prime}}^k\left(E\right)$ and $H^{2n-k}_{D^{\prime\prime}}\left(E^*\right)$.}
\end{lemma}
\begin{proof}
  $\bar{\star}_E:\Omega^k\left(E\right)\to\Omega^{2n-k}\left(E^*\right)$ and $\bar{\star}_{E^*}:\Omega^k\left(E^*\right)\to\Omega^{2n-k}\left(E\right)$ denote the Hodge's star operator on $E$ and $E^*$, respectively. If we decompose $D^{\prime\prime}$ and $D_{E^*}^{\prime\prime}$ as $D^{\prime\prime}=\bar{\partial}^E+\theta^E$ and $D_{E^*}^{\prime\prime}=\bar{\partial}^{E^*}+\theta^{E^*}$, the adjoint operator of $\bar{\partial}^E$ is written by $\left(\bar{\partial}^E\right)^*=-\bar{\star}_{E^*}\circ\bar{\partial}^{E^*}\circ\bar{\star}_E$. (See \cite[p.14]{Simpson2} for the definition of $\theta^{E^*}$.)
  
  In what follows, we prove the following identity.
  \[\left(\theta^E\wedge\right)^*=-\bar{\star}_{E^*}\circ\left(\theta^{E^*}\wedge\right)\circ\bar{\star}_E.\]
  For $\alpha\in\Omega^k\left(E\right),\beta\in\Omega^{k+1}\left(E\right)$, we have
  \begin{align}
    \int_Mh_{E\otimes\Lambda^{k+1}M}\left(\theta\wedge\alpha,\beta\right)vol_g=&\int_M\langle\theta\wedge\alpha,\bar{\star}_E\beta\rangle\\
    =&-\left(-1\right)^k\int_M\langle\alpha,\theta^{E^*}\wedge\bar{\star}_E\beta\rangle\\
    =&\int_Mh_{E\otimes\Lambda^{k+1}M}\left(\alpha,-\left(\bar{\star}_{E^*}\circ\left(\theta^{E^*}\wedge\right)\circ\bar{\star}_E\right)\beta\right) vol_g,
  \end{align}
  where $\langle\cdot,\cdot\rangle$ is a pairing. Therefore we obtain the desired identity.

  From the above discussion, the adjoint operator of $D^{\prime\prime}$ is written by $\left(D^{\prime\prime}\right)^*=-\bar{\star}_{E^*}\circ D_{E^*}^{\prime\prime}\circ\bar{\star}_E$. Let $\Delta^{D^{\prime\prime}}$ be a Laplacian induced by $D^{\prime\prime}$. It is an elliptic operator and $H_{D^{\prime\prime}}^k\left(E\right)$ coincides with the kernel of $\Delta^{D^{\prime\prime}}$. Thus we obatin $\bar{\star}_E\circ \Delta^{D^{\prime\prime}}=\Delta^{D_{E^*}^{\prime\prime}}\circ\bar{\star}_{E^*}$ and $\bar{\star}_E$ defines an anti-linear isomorphism between $H_{D^{\prime\prime}}^k\left(E\right)$ and $H_{D_{E^*}^{\prime\prime}}^{2n-k}\left(E^*\right)$. This completes the proof.
\end{proof}
Since the dual bundle of $\textrm{End}\left(E\right)$ is isomorphic to itself, by the above lemma, we have $H_{D^{\prime\prime}}^k\left(\mathscr{C}^*\right)=H_{D^{\prime\prime}}^{2n-k}\left(\mathscr{C}^*\right)$.
\begin{theorem}
  
  \textit{If $\left(M,g\right)$ is a compact K\"{a}hler surface, the real dimension of moduli space $\widetilde{\mathscr{M}}_{\textrm{HYM}}$ is given by $2+b_2-r^2\chi\left(M\right)$, where $b_2$ is the second betti bumber and $\chi\left(M\right)$ is the Euler number of $M$.}
\end{theorem}
\begin{proof}
  For a point $\left[D^{\prime\prime}\right]\in\widetilde{\mathscr{M}}_{\textrm{HYM}}$, the topological index of the complex $\left(\mathscr{C}^*\right)$ is given by
  \[\textrm{ind}_T\left(\mathscr{C^*}\right)=\int_M\frac{\textrm{ch}\left(\sum_{j=1}^{2n}\Lambda^jM\otimes\textrm{End}\left(E\right)\right)\textrm{Td}\left(TM\otimes\mathbb{C}\right)}{e\left(TM\right)}=r^2\chi\left(M\right).\]
  Therefore by Atiyah-Singer index theorem, we have
  \[\dim H_{D^{\prime\prime}}^0-\dim H_{D^{\prime\prime}}^1+\dim H_{D^{\prime\prime}}^2-\dim H_{D^{\prime\prime}}^3-\dim H_{D^{\prime\prime}}^4=r^2\chi\left(M\right),\]
  where $H_{D^{\prime\prime}}^k=H_{D^{\prime\prime}}^k\left(\mathscr{C}^*\right)$. Since $D^{\prime\prime}$ satisfies \eqref{eq:4.1} and by the Lemma \ref{Lem:4.3}, we have
  \[\dim H_{D^{\prime\prime}}^0=\dim H_{D^{\prime\prime}}^4=1,\quad \dim H_{D^{\prime\prime}}^1=\dim H_{D^{\prime\prime}}^3,\quad \dim H_{D^{\prime\prime}}^2=b_2.\]
  Thereofore we obtain $2\dim H_{D^{\prime\prime}}^1=2+b_2-r^2\chi\left(M\right)$, this completes the proof.
\end{proof}
\begin{remark}
  We can also compute the dimension of moduli space when $M$ is not a surface. But if the dimension of $M$ is $3$ or higher, the dimension of moduli space may vary between its connected components. If $M$ is a Riemann surface, by a same computation, we have $\dim H_{D^{\prime\prime}}^0-\dim H_{D^{\prime\prime}}^1+\dim H_{D^{\prime\prime}}^2=r^2\chi\left(M\right)$ and the real dimension is given by $4+4r^2\left(g-2\right)$, where $g$ is the genus of $M$. This result is also stated in \cite[Theorem 3.4, p.257]{Wells}.
\end{remark}
\begin{remark}
  If $M$ is a surface, the HYM equation is equivalent to the Kapustin-Witten equation \cite{Tanaka}. If the structure group of vector bundle $E$ is $SU\left(2\right)$, the dimension of the smooth locus of the moduli space of the solutions to Kapustin-Witten equation is given by $-3\chi\left(M\right)$, which agrees with the result obtained in the above theorem when the structure group is replaced by $SU\left(2\right)$ \cite{Liu}.
\end{remark}
\section{A relationship between the moduli space of simple Higgs structures and the moduli space of irreducible HYM connections}
Let $\left(M,g\right)$ be a compact K\"{a}hler manifold and $\left(E,h\right)$ be a smooth Hermitian vector bundle over $M$. For $g\in\mathcal{G}\left(E\right)$, we define an Hermitian metric $h\cdot g$ on $E$ by, for $s,t\in\Omega^0\left(E\right)$,
\begin{equation}
  \left(h\cdot g\right)\left(s,t\right):=h\bigl(g\left(s\right),g\left(t\right)\bigr). \label{eq:6.1}
\end{equation}
For $D^{\prime\prime}\in\mathscr{D}^{\prime\prime}\left(E\right)$, $D_{D^{\prime\prime},h}$ denotes the connection on $E$ induced by $D^{\prime\prime}$. The space of connections $\mathscr{A}\left(E\right)$ has a right-action of $\mathcal{G}\left(E\right)$ via \eqref{eq:2.4}. To explicitly indicate $h$, we write this action as $\cdot_h$. Based on these notations, the following can be readly shown:
\begin{proposition}\label{prop:6.2}
  \textit{For} $D^{\prime\prime}\in\mathscr{D}^{\prime\prime}\left(E\right)$ \textit{and} $g\in\mathcal{G}\left(E\right)$, \textit{we have}
  \begin{equation}
    g^{-1}\circ D_{D^{\prime\prime},h\cdot g^{-1}}\circ g=D_{D^{\prime\prime},h}\cdot_h g=D_{D^{\prime\prime}\cdot g,h}.
  \end{equation}
\end{proposition}
From this, the following holds.
\begin{proposition}\label{prop:6.3}
  \textit{If} $D^{\prime\prime}\in\mathscr{H}^{\prime\prime}\left(E\right)$ \textit{and} $g\in\mathcal{G}\left(E\right)$, \textit{then the following are equivalent:}
  \begin{enumerate}[$\left(1\right)$]
    \item $D_{D^{\prime\prime},h}\cdot_h g=D_{D^{\prime\prime}\cdot g,h}\in\mathscr{Y}\left(E,h\right)$.
    \item $D_{D^{\prime\prime},h\cdot g^{-1}}\in\mathscr{Y}\left(E,h\cdot g^{-1}\right)$.
  \end{enumerate}
\end{proposition}
Let $\textrm{Herm}^+\left(E\right)$ be the set of Hermitian metrics on $E$ and we define the set $\mathcal{Y}\left(E\right)$ as
\begin{equation}
  \mathcal{Y}\left(E\right):=\left\{\left(D^{\prime\prime},h\right)\in\mathscr{H}^{\prime\prime}\left(E\right)\times\textrm{Herm}^+\left(E\right)\middle|
  \begin{gathered}
    \sqrt{-1}\Lambda R\left(D_{D^{\prime\prime},h}\right)=c\,\textrm{id}_E,\\
    c=\frac{2\pi\mu\left(E\right)}{\left(n-1\right)!\cdot\textrm{Vol}\left(M,g\right)}
  \end{gathered}\right\}.
\end{equation}
We define the right-action of $\mathcal{G}\left(E\right)$ on $\mathcal{Y}\left(E\right)$ by, for $g\in\mathcal{G}\left(E\right)$ and $\left(D^{\prime\prime},h\right)\in\mathcal{Y}\left(E\right)$,
\begin{equation}
  \left(D^{\prime\prime},h\right)\cdot g:=\left(D^{\prime\prime}\cdot g,h\cdot g\right).
\end{equation}
This is well-defined by Proposition \ref{prop:6.3}. Then the following holds.
\begin{proposition}(Analogue of Proposition $3.7$ in Bradlow \cite{Bradlow})

  \textit{There exists a bijection between} $\mathscr{M}_{\textrm{HYM}}=\mathscr{Y}\left(E,h\right)/\mathcal{G}\left(E,h\right)$ \textit{and} $\mathcal{Y}\left(E\right)/\mathcal{G}\left(E\right)$.
\end{proposition}
\begin{proof}
  Consider a mapping
  \begin{equation}
    \mathscr{Y}\left(E,h\right)/\mathcal{G}\left(E,h\right)\ni\left[D_{D^{\prime\prime},h}\right]\mapsto\left[D^{\prime\prime},h\right]\in\mathcal{Y}\left(E\right)/\mathcal{G}\left(E\right).
  \end{equation}
  For $D_{D_1^{\prime\prime},h},D_{D_2^{\prime\prime},h}\in\mathscr{Y}\left(E,h\right)$, if $\left[D_{D_1^{\prime\prime},h}\right]=\left[D_{D_2^{\prime\prime},h}\right]$, then there exists $g\in\mathcal{G}\left(E,h\right)$ such that $D_{D_2^{\prime\prime},h}=D_{D_1^{\prime\prime},h}\cdot g$ holds, which implies that $D_2^{\prime\prime}=D_1^{\prime\prime}\cdot g$. Moreover since $g\in\mathcal{G}\left(E,h\right)$ satisfies $h\cdot g=h$, we have $\left[D_1^{\prime\prime},h\right]=\left[D_2^{\prime\prime},h\right]$. Therefore this mapping is well-defined.

  Conversely, for $\left[D^{\prime\prime},k\right]\in\mathcal{Y}\left(E\right)/\mathcal{G}\left(E\right)$, there exists $g\in\mathcal{G}\left(E\right)$ such that $k=h\cdot g$ so consider a mapping
  \begin{equation}
    \mathcal{Y}\left(E\right)/\mathcal{G}\left(E\right)\ni\left[D^{\prime\prime},k\right] \mapsto \left[D_{D^{\prime\prime}\cdot g^{-1},h}\right]\in\mathscr{Y}\left(E,h\right)/\mathcal{G}\left(E,h\right).
  \end{equation}
  For $\left(D_1^{\prime\prime},k_1\right),\left(D_2^{\prime\prime},k_2\right)\in\mathcal{Y}\left(E\right)$, if $\left[D_1^{\prime\prime},k_1\right]=\left[D_2^{\prime\prime},k_2\right]$, then there exists $g\in\mathcal{G}\left(E\right)$ such that $\left(D_2^{\prime\prime},k_2\right)=\left(D_1^{\prime\prime}\cdot g,k_1\cdot g\right)$ holds. Also there exist $g_1,g_2\in\mathcal{G}\left(E\right)$ such that $k_1=h\cdot g_1,k_2=h\cdot g_2$ and $g_1\circ g\circ g_2^{-1}\in\mathcal{G}\left(E,h\right)$. Therefore applying Proposition \ref{prop:6.2} multiple times, we have $D_{D_2^{\prime\prime}\cdot g_2^{-1},h}=g_2\circ g^{-1}\circ g_1^{-1}\circ D_{D_1^{\prime\prime}\cdot g_1^{-1},h}\circ g_1\circ g\circ g_2^{-1}$. From this, we obtain $\left[D_{D_1^{\prime\prime}\cdot g_1^{-1},h}\right]=\left[D_{D_2^{\prime\prime}\cdot g_2^{-1},h}\right]$. Therefore this mapping is well-defined.
  
  These mappings are inverse to each other, which completes the proof.
\end{proof}

We define a subset $\widetilde{\mathscr{H}}^{\prime\prime}\left(E\right)$ of $\mathscr{H}^{\prime\prime}\left(E\right)$ as
\begin{equation}
  \widetilde{\mathscr{H}}^{\prime\prime}\left(E\right)=\left\{D^{\prime\prime}\in\mathscr{H}^{\prime\prime}\left(E\right)\middle|\left(E,D^{\prime\prime}\right)\text{ is simple}\right\},
\end{equation}
and similarly define $\widetilde{\mathcal{Y}}\left(E\right)$ as
\begin{equation}
  \widetilde{\mathcal{Y}}\left(E\right)=\left\{\left(D^{\prime\prime},h\right)\in\mathcal{Y}\left(E\right)\middle|D^{\prime\prime}\in\widetilde{\mathscr{H}}^{\prime\prime}\left(E\right)\right\}.
\end{equation}
Then $\mathcal{G}\left(E\right)$ preserves $\widetilde{\mathcal{Y}}\left(E\right)$ and $\widetilde{\mathscr{H}}^{\prime\prime}\left(E\right)$. Furthermore we have the following:
\begin{theorem}(Analogous to L\"{u}ble-Teleman \cite[Proposition $2.2.2$]{Telemann})\label{thm:7.4}

  \textit{For} $\left(D_1^{\prime\prime},h_1\right),\left(D_1^{\prime\prime},h_2\right)\in\widetilde{\mathcal{Y}}\left(E\right)$, \textit{there exists} $a\in\mathbb{R}_{>0}$ \textit{such that} $h_2=ah_1$.
\end{theorem}
\begin{proof}
  The right action of $\mathcal{G}\left(E\right)$ on $\text{Herm}^+\left(E\right)$, defined by \eqref{eq:6.1}, is transitive. In particular, there exists a unique positive definite Hermitian transformation $f\in\Omega^0\bigl(\textrm{Herm}\left(E,h_1\right)\bigr)$ satisfying $h_2\left(\cdot,\cdot\right)=h_1\bigl(f\left(\cdot\right),\cdot\bigr)$. $g\in\Omega^0\bigl(\textrm{Herm}\left(E,h_1\right)\bigr)$ denotes the square root of $f$. Then we have $h_2=h_1\cdot g$. Also, if we define $D_2^{\prime\prime}:=D_1^{\prime\prime}\cdot g^{-1}$, then we have
  \begin{equation}
    D_{\textrm{Hom}\bigl(\left(E,D_1^{\prime\prime}\right),\left(E,D_2^{\prime\prime}\right)\bigr)}^{\prime\prime}g=D_2^{\prime\prime}\circ g-g\circ D_1^{\prime\prime}=0.
  \end{equation}
  Let $D_{D_1^{\prime\prime},h_1}=D_1^\prime+D_1^{\prime\prime}$, then from Proposition \ref{prop:6.2}, we have
  \begin{equation}
    D_{D_1^{\prime\prime},h_2}=D_{D_1^{\prime\prime},h_1\cdot g}=g^{-2}\circ D_1^\prime\circ g^2+D_1^{\prime\prime}.
  \end{equation}
  Therefore we obtain $R\left(D_{D_2^{\prime\prime},h_1}\right)=g\circ R\left(D_{D_1^{\prime\prime},h_2}\right)\circ g^{-1}$. Thus $D_{D_2^{\prime\prime},h_1}$ is an HYM connection with HYM factor $c$ and this implies that $D_{\text{Hom}\bigl(\left(E,D_1^{\prime\prime}\right),\left(E,D_2^{\prime\prime}\right)\bigr)}$ is an HYM connection with HYM factor $0$. In other words, since the mean curvature form of $D_{\text{Hom}\bigl(\left(E,D_1^{\prime\prime}\right),\left(E,D_2^{\prime\prime}\right)\bigr)}$ is semi-negative definite, we have $D_{\text{Hom}\bigl(\left(E,D_1^{\prime\prime}\right),\left(E,D_2^{\prime\prime}\right)\bigr)}^\prime g=0$. Let $D_{D_2^{\prime\prime},h_1}=D_2^\prime+D_2^{\prime\prime}$, then we have
  \begin{equation}
    0=D_{\text{Hom}\bigl(\left(E,D_1^{\prime\prime}\right),\left(E,D_2^{\prime\prime}\right)\bigr)}^\prime g=g^{-1}\circ D_{\textrm{End}\left(E,D_1^{\prime\prime}\right)}^\prime f.
  \end{equation}
  Therefore we have $D_{\textrm{End}\left(E,D_1^{\prime\prime}\right)}^\prime f=0$. Since $f\in\Omega^0\bigl(\textrm{Herm}\left(E,h_1\right)\bigr)$, we obtain 
  \begin{equation}
    D_{\text{End}\left(E,D_1^{\prime\prime}\right)}^{\prime\prime}f=0
  \end{equation}
  and since $\left(E,D_1^{\prime\prime}\right)$ is simple, there exists $a\in\mathbb{R}_{>0}$ satisfying $f=a\,\text{id}_E$. This completes the proof.
\end{proof}
In the above proof, we used the following theorem.
\begin{theorem}
  \textit{Let} $\left(M,g\right)$ \textit{be a compact K\"{a}hler manifold,} $\left(E,D^{\prime\prime},h\right)$ \textit{be a Hermitian Higgs bundle over} $M$ and $D$ be \textit{be the HS connection on} $\left(E,D^{\prime\prime},h\right)$. \textit{Let} $\widehat{K}\left(D\right)$ \textit{denote the mean curvature form of} $D$. \textit{If} $\widehat{K}\left(D\right)$ \textit{is semi-negative definite, that is, for any} $s\in\Omega^0\left(E\right),\widehat{K}\left(D\right)\left(s,s\right) \leq 0$, \textit{then any} $s \in \Omega^0\left(E\right)$ \textit{satisfying} $D^{\prime\prime}s=0$ \textit{satisfies} $D^\prime s=0$.
  \end{theorem}
  We can prove this easily by using K\"{a}hler identities and Bochner's formula and we omit the proof.
\begin{remark}
  Let $\left(E,D^{\prime\prime}\right)$ be a simple Higgs bundle over $M$. Then from Theorem \ref{thm:7.4}, if an HYM metric on $\left(E,D^{\prime\prime}\right)$ exists, it is unique up to positive real scalar multiplication.
\end{remark}
Moreover the following holds.
\begin{proposition}\label{prop:7.6}
  \textit{The natural map}
\begin{equation}
  \widetilde{\mathcal{Y}}\left(E\right)/\mathcal{G}\left(E\right)\ni\left[D^{\prime\prime},h\right]\mapsto\left[D^{\prime\prime}\right]\in\widetilde{\mathscr{H}}^{\prime\prime}\left(E\right)/\mathcal{G}\left(E\right)
\end{equation}
\textit{is injective}.
\end{proposition}
\begin{proof}
  Since the injectivity of this mapping follows from Theorem \ref{thm:7.4}, we will show this map is well-defined. For $\left(D_1^{\prime\prime},h_1\right),\left(D_2^{\prime\prime},h_2\right)\in\mathcal{Y}\left(E\right)$, if $\left[D_1^{\prime\prime},h_1\right]=\left[D_2^{\prime\prime},h_2\right]$, then there exists $g\in\mathcal{G}\left(E\right)$ satisfying $\left(D_2^{\prime\prime},h_2\right)=\left(D_1^{\prime\prime},h_1\right)\cdot g$. Therefore we have $D_2^{\prime\prime}=D_1^{\prime\prime}\cdot g$ so $\left[D_1^{\prime\prime}\right]=\left[D_2^{\prime\prime}\right]$ holds. Hence it is well-defined. 
\end{proof}
Define a subset $\widetilde{\mathscr{Y}}\left(E,h\right)$ of $\mathscr{Y}\left(E,h\right)$ as
\begin{equation}
  \widetilde{\mathscr{Y}}\left(E,h\right)=\left\{D\in\mathscr{Y}\left(E,h\right)\middle|D\textrm{ is irreducible}\right\},
\end{equation}
then the right-action of $\mathcal{G}\left(E,h\right)$ on $\mathscr{Y}\left(E,h\right)$ preserves $\widetilde{\mathscr{Y}}\left(E,h\right)$ and there exists a bijection between $\widetilde{\mathscr{Y}}\left(E,h\right)/\mathcal{G}\left(E,h\right)$ and $\widetilde{\mathcal{Y}}\left(E\right)/\mathcal{G}\left(E\right)$. We denote $\widetilde{\mathscr{H}}^{\prime\prime}\left(E\right)/\mathcal{G}\left(E\right)$ and $\widetilde{\mathscr{Y}}\left(E,h\right)/\mathcal{G}\left(E,h\right)$ by $\widetilde{\mathscr{M}}_{\textrm{Higgs}}$ and $\widetilde{\mathscr{M}}_{\textrm{HYM}}$, respectively. From the above bijection and Proposition \ref{prop:7.6}, there is an injection from $\widetilde{\mathscr{M}}_{\textrm{HYM}}$ into $\widetilde{\mathscr{M}}_{\textrm{Higgs}}$. We denotes it by $f$.

We conclude preparations with the following lemma. The proof is straightforward and thus we omit the proof.
\begin{lemma}\label{lem:6.9}
  \textit{Let} $\left(E,D^{\prime\prime},h\right)$ \textit{be a Hermitian Higgs bundle over} $M$ \textit{and} $D_{D^{\prime\prime},h}\in\mathscr{H}\left(E,h\right)$ \textit{satisfy the condition that there exists a smooth function} $\varphi:M\to\mathbb{R}$ \textit{satisfying} $K\left(D_{D^{\prime\prime},h}\right)=\varphi\,\text{id}_E$. \textit{Then there exists a smooth positive-valued function} $a:M\to\mathbb{R}$ \textit{such that} $D_{D^{\prime\prime},ah}\in\mathscr{Y}\left(E,ah\right)$.
\end{lemma}
Based on these, we prove the following theorem.
\begin{theorem}
  $f:\widetilde{\mathscr{M}}_{\textrm{HYM}}\to\widetilde{\mathscr{M}}_{\textrm{Higgs}}$ \textit{is an open mapping}.
\end{theorem}
\begin{proof}
  Let $D=D_{D^{\prime\prime},h}\in\widetilde{\mathscr{Y}}\left(E,h\right)$ and $s>n$. We will show that for $D_1^{\prime\prime}\in\mathscr{H}^{\prime\prime}\left(E\right)$ sufficiently close to $D^{\prime\prime}$, there exist $D_{D_2^{\prime\prime},h}\in\widetilde{\mathscr{M}}_{\textrm{HYM}}$ and $g^\prime\in\mathcal{G}\left(E\right)$ such that $D_2^{\prime\prime}=g^\prime\cdot D_1^{\prime\prime}$. Let
  \begin{equation}
    P=\left\{g\in\mathcal{G}\left(E\right)\middle|g\text{ is positive definite},g\in\Omega^0\bigl(\textrm{Herm}\left(E,h\right)\bigr),\det g=1\right\},
  \end{equation}
  then the tangent space of $P$ at $\textrm{id}_E\in P$, denoted by $T_{\textrm{id}_E}P=:Q$, is given by
  \begin{equation}
    Q=\left\{a\in\Omega^0\bigl(\textrm{Herm}\left(E,h\right)\bigr)\middle|\tr a=0\right\}.
  \end{equation}

  In general, for any $D^{\prime\prime}\in\mathscr{D}^{\prime\prime}\left(E\right)$, we have $D_{D^{\prime\prime},h}\in\mathscr{A}\left(E\right)$, and for $g\in P$, we obtain
  \begin{equation}
    D_{D^{\prime\prime}\cdot g,h}=D_{D^{\prime\prime},h}\cdot_hg=g^*\circ D^\prime\circ\left(g^{-1}\right)^*+g^{-1}\circ D^{\prime\prime}\circ g=g\circ D^\prime\circ g^{-1}+g^{-1}\circ D^{\prime\prime}\circ g.
  \end{equation}
  We define $\displaystyle K^0\left(D_{D^{\prime\prime}\cdot g,h}\right):=K\left(D_{D^{\prime\prime} \cdot g,h}\right)-\frac{1}{r}\tr\bigl(K\left(D_{D^{\prime\prime}\cdot g,h}\right)\bigr)\textrm{id}_E\in\Omega^0\bigl(\textrm{End}\left(E\right)\bigr)$, where $r=\textrm{rank}\left(E\right)$ and $K_{\textrm{Herm}}^0\left(D_{D^{\prime\prime}\cdot g,h}\right)$ denotes the $\textrm{Herm}\left(E,h\right)$-component of $K^0\left(D_{D^{\prime\prime}\cdot g,h}\right)$. Then we have $K_{\textrm{Herm}}^0\left(D_{D^{\prime\prime}\cdot g,h}\right)\in Q$. Therefore we define a mapping $F:\mathscr{A}\left(E\right)\times P\to Q$ by, for $\left(D_{D^{\prime\prime},h},g\right)\in\mathscr{A}\left(E\right)\times P,F\left(D_{D^{\prime\prime},h},g\right)=K_{\textrm{Herm}}^0\left(D_{D^{\prime\prime}\cdot g,h}\right)$
  and extend it to $F:W^s\bigl(\mathscr{A}\left(E\right)\bigr)\times W^{s+1}\left(P\right)\to W^{s-1}\left(Q\right)$. For the initially given $D\in\widetilde{\mathscr{Y}}\left(E,h\right)$, we calculate the value at $\left(0,a\right)$ of the differential $\left(dF\right)_{\left(D_{D^{\prime\prime},h},\textrm{id}_E\right)}:W^s\left(\mathscr{C}^1\right)\times W^{s+1}\left(Q\right)\to W^{s-1}\left(Q\right)$. We have
  \begin{equation}
    \left.\frac{d}{dt}\right|_{t=0}K\left(D_{D^{\prime\prime}\cdot e^{at},h}\right)=-\Delta a
  \end{equation}
  and this is trace-free, where $\Delta=D_{D_{\textrm{End}\left(E\right)}^{\prime\prime},h_{\textrm{End}\left(E\right)}}^*\circ D_{D_{\textrm{End}\left(E\right)}^{\prime\prime},h_{\textrm{End}\left(E\right)}}$. Moreover if we decompose $D_{D_{\textrm{End}\left(E\right)}^{\prime\prime},h_{\textrm{End}\left(E\right)}}$ into unitary connection part and Hermitian transformation on $\textrm{End}\left(E\right)$-valued $1$-form part as
  \begin{equation}
    D_{D_{\textrm{End}\left(E\right)}^{\prime\prime},h_{\textrm{End}\left(E\right)}}=\nabla^{\textrm{End}\left(E\right)}+\Theta^{\textrm{End}\left(E\right)},
  \end{equation}
  then the $\textrm{Herm}\left(E,h\right)$-component of $\Delta a$ is given by
  \begin{equation}
    \left(\nabla^{\textrm{End}\left(E\right)}\right)^*\nabla^{\textrm{End}\left(E\right)}a+\left(\Theta^{\textrm{End}\left(E\right)}\wedge\right)^*\Theta^{\textrm{End}\left(E\right)}a=\Delta^{\nabla^{\textrm{End}\left(E\right)}}a+\left(\Theta^{\textrm{End}\left(E\right)}\wedge\right)^*\Theta^{\textrm{End}\left(E\right)} a.
  \end{equation}
  Therefore we have $\left(dF\right)_{\left(D_{D^{\prime\prime},h},\textrm{id}_E\right)}\left(0,a\right)=-\left(\Delta^{\nabla^{\textrm{End}\left(E\right)}}a+\left(\Theta^{\textrm{End}\left(E\right)}\wedge\right)^*\Theta^{\textrm{End}\left(E\right)}a\right)$. We view it as the mapping $\left(dF\right)_{\left(D_{D^{\prime\prime},h},\textrm{id}_E\right)}:W^{s+1}\left(Q\right)\to W^{s-1}\left(Q\right)$. Then since $D^{\prime\prime}$ is simple, it is a linear isomorphism. Furthermore since $F\left(D_{D^{\prime\prime},h},\textrm{id}_E\right)=0$, for $D_1^{\prime\prime}$ sufficiently close to $D^{\prime\prime}\in\mathscr{H}^{\prime\prime}\left(E\right)$, by the implicit function theorem, there exists a unique $g\in W^{s+1}\left(P\right)$ sufficiently close to $\textrm{id}_E\in P$ satisfying $F\left(D_{D_1^{\prime\prime},h},g\right)=0$. We write $D_1^{\prime\prime}$ as $D_1^{\prime\prime}=D^{\prime\prime}+\alpha$ for some $\alpha\in\Omega^1\bigl(\textrm{End}\left(E\right)\bigr)$. Then we have $D_1^{\prime\prime}\cdot g=D^{\prime\prime}+g^{-1}\circ D^{\prime\prime}g+g^{-1}\circ\alpha\circ g$. Since $D_1^{\prime\prime}\cdot g\in W^s\bigl(\mathscr{H}^{\prime\prime}\left(E\right)\bigr)$, if we define $\beta\in W^s\left(\mathscr{C}^1\right)$ by $\beta=g^{-1}\circ D_{\textrm{End}\left(E\right)}^{\prime\prime}g+g^{-1}\circ\alpha\circ g$, then we have
  \begin{align}
    R\left(D_{D_1^{\prime\prime}\cdot g,h}\right)=D^\prime\circ D^{\prime\prime}+D^{\prime\prime}\circ D^\prime+D^\prime\beta+D^{\prime\prime}\left(\beta^\prime-\beta^{\prime\prime}\right)^*+\left[\beta,\left(\beta^\prime-\beta^{\prime\prime}\right)^*\right].
  \end{align}
  Thus we obtain
  \begin{align}
    K^0\left(D_{D_1^{\prime\prime}\cdot g,h}\right)=&\sqrt{-1}\Lambda\left(D^\prime\beta+D^{\prime\prime}\left(\beta^\prime-\beta^{\prime\prime}\right)^*+\left[\beta,\left(\beta^\prime-\beta^{\prime\prime}\right)^*\right]\right)\\
    &-\frac{1}{r}\sqrt{-1}\Lambda\left(\partial\tr\beta+\bar{\partial}\tr\left(\beta^\prime-\beta^{\prime\prime}\right)^*\right)\,\textrm{id}_E.
  \end{align}
  Furthermore since $D_1^{\prime\prime}\in\mathscr{H}^{\prime\prime}\left(E\right)$, we have $K_{\textrm{Herm}}^0\left(D_{D_1^{\prime\prime}\cdot g,h}\right)=K^0\left(D_{D_1^{\prime\prime}\cdot g,h}\right)$. Therefore $g\in W^{s+1}\left(P\right)$ is a solution to the following equation:
  \begin{align}
    \left\{
    \begin{aligned}
        &\sqrt{-1}\Lambda\left(D_{\textrm{End}\left(E\right)}^\prime\beta+D_{\textrm{End}\left(E\right)}^{\prime\prime}\left(\beta^\prime-\beta^{\prime\prime}\right)^*+\left[\beta,\left(\beta^\prime-\beta^{\prime\prime}\right)^*\right]\right)\\
        &\quad\quad\quad-\frac{1}{r}\sqrt{-1}\Lambda\left(\partial\tr\beta+\bar{\partial}\tr\left(\beta^\prime-\beta^{\prime\prime}\right)^*\right)\,\textrm{id}_E=0, \\
        &\beta=g^{-1}\circ D_{\textrm{End}\left(E\right)}^{\prime\prime}g+g^{-1}\circ\alpha\circ g.
    \end{aligned}
    \right.
  \end{align}
  From this equation and a straightforward calculation, it can be shown that $g$ is smooth. Since $K^0\left(D_{D_1^{\prime\prime}\cdot g, h}\right)=0$, there exists a smooth function $\varphi:M\to\mathbb{R}$ satisfying $K\left(D_{D_1^{\prime\prime}\cdot g,h}\right)=\varphi\,\textrm{id}_E$. Therefore by Theorem \ref{lem:6.9}, there exists a smooth positive definite function $a:M\to\mathbb{R}$ such that $D_{D_1^{\prime\prime}\cdot g,h\cdot a\,\textrm{id}_E}=D_{D_1^{\prime\prime}\cdot g,a^2h}\in\mathscr{Y}\left(E,a^2h\right)=\mathscr{Y}\left(E,h\cdot a\,\textrm{id}_E\right)$. Thus by Theorem \ref{prop:6.3}, we obtain $D_{D_1^{\prime\prime}\cdot\left(g\circ\frac{1}{a}\,\textrm{id}_E\right),h}\in\mathscr{Y}\left(E,h\right)$ Therefore defining $\displaystyle D_2^{\prime\prime}:=D_1^{\prime\prime}\cdot\left(g\circ\frac{1}{a}\,\textrm{id}_E\right)$ and $\displaystyle g^\prime:=g\circ\frac{1}{a}\,\textrm{id}_E$, this completes the proof.
\end{proof}
\appendix
\section*{Appendix: Infinitesimal deformations and the deformation complex of Higgs structures}
Let $\left(M,g\right)$ be a compact Kähler manifold and $E$ be a smooth complex vector bundle over $M$. We formally calculate the tangent space at $\left[D^{\prime\prime}\right]\in\mathscr{M}_{\textrm{Higgs}}$. Let
\begin{equation}
  D^{\prime\prime}_t=D^{\prime\prime}+\alpha_t,\quad\left|t\right|<\varepsilon\ll 1
\end{equation}
be a curve in $\mathscr{H}^{\prime\prime}\left(E\right)$ where $\alpha_t\in\Omega^1\bigl(\textrm{End}\left(E\right)\bigr)$ and $\alpha_0=0$. Since $D^{\prime\prime}_t\in\mathscr{H}^{\prime\prime}\left(E\right)$, we have $D_t^{\prime\prime}\circ D_t^{\prime\prime}=D_{\textrm{End}\left(E\right)}^{\prime\prime}\alpha_t+\alpha_t\wedge\alpha_t=0$. Differentiating both sides at $t=0$, we obtain
\begin{equation}
  D_{\textrm{End}\left(E\right)}^{\prime\prime}\alpha=0,\textrm{ where }\alpha=\left.\frac{d}{dt}\right|_{t=0}\alpha_t.
\end{equation}
Moreover if $D^{\prime\prime}_t$ can be expressed by a one-parameter family of gauge transformations $\left(g_t\right)_{\left|t\right|<\varepsilon}$ of $E$ as $D^{\prime\prime}_t=g_t^{-1}\circ D^{\prime\prime}\circ g_t$, then differentiating both sides at $t=0$, we obtain
\begin{equation}
  \alpha=D_{\textrm{End}\left(E\right)}^{\prime\prime}g,\textrm{ where }g=\left.\frac{d}{dt}\right|_{t=0}g_t.
\end{equation}
From these, if $\mathscr{M}_{\textrm{Higgs}}$ is a manifold, the ``tangent space'' at $\left[D^{\prime\prime}\right]$ can be considered as
\begin{equation}
  T_{\left[D^{\prime\prime}\right]}\mathscr{M}_{\textrm{Higgs}}=\frac{\left\{\alpha\in\Omega^1\bigl(\textrm{End}\left(E\right)\bigr)\middle|D_{\textrm{End}\left(E\right)}^{\prime\prime}\alpha=0\right\}}{\left\{D_{\textrm{End}\left(E\right)}^{\prime\prime}g\middle|g\in\Omega^0\bigl(\textrm{End}\left(E\right)\bigr)\right\}}.
\end{equation}
This can be formulated as the cohomology group of a certain elliptic complex as follows. Fix a Hermitian metric $h$ on $E$ and consider the following sequence for $D^{\prime\prime}\in\mathscr{H}^{\prime\prime}\left(E\right)$:
\begin{equation}
  \xymatrix{\left(\mathscr{C}^*\right):0\ar[r]&\mathscr{C}^0\ar[r]^-{D_{\textrm{End}\left(E\right)}^{\prime\prime}}&\mathscr{C}^1\ar[r]^-{D_{\textrm{End}\left(E\right)}^{\prime\prime}}&\mathscr{C}^2\ar[r]^-{D_{\textrm{End}\left(E\right)}^{\prime\prime}}&\cdots\ar[r]^-{D_{\textrm{End}\left(E\right)}^{\prime\prime}}&\mathscr{C}^{2n}\ar[r]&0}
\end{equation}
Since $D^{\prime\prime}\in\mathscr{H}^{\prime\prime}\left(E\right)$, this sequence forms a complex and is called the Dolbeault-Higgs complex \cite[$\left(9.2.1\right)$]{Fujiki}.
\begin{proposition*}\cite{Fujiki}
  \textit{The Dolbeault-Higgs complex} $\left(\mathscr{C}^*\right)$ \textit{is an elliptic complex}.
\end{proposition*}
We can prove this in the same way on the case of Dolbeault complex and we omit the proof. Let $H_{D^{\prime\prime}}^k\left(\mathscr{C}^*\right)$ be the $k$-th cohomology group. Then $T_{\left[D^{\prime\prime}\right]}\mathscr{M}_{\textrm{Higgs}}\simeq H_{D^{\prime\prime}}^1\left(\mathscr{C}^*\right)$ holds. Let $(\widetilde{\mathscr{C}}^*)$ be a subcomplex of $\left(\mathscr{C}^*\right)$ formed by $\widetilde{\mathscr{C}}^k$, then it is also elliptic. Let $\widetilde{H}_{D^{\prime\prime}}^k\left(\mathscr{C}^*\right)$ be the $k$-th cohomology group of $(\widetilde{\mathscr{C}}^*)$, then since $\textrm{End}\left(E\right) = \textrm{End}^0\left(E\right)\oplus\mathbb{C}\,\textrm{id}_E$, we have
\renewcommand{\theequation}{\arabic{equation}}
\setcounter{equation}{0}
\begin{equation}
  H_{D^{\prime\prime}}^k\left(\mathscr{C}^*\right)=\widetilde{H}_{D^{\prime\prime}}^k\left(\mathscr{C}^*\right)+\bigoplus_{p+q=k} H_{\bar{\partial}}^{p,q}\left(M\right)=\widetilde{H}_{D^{\prime\prime}}^k\left(\mathscr{C}^*\right)+H^k\left(M,\mathbb{C}\right). \label{eq:ap}
\end{equation}
\section*{Acknowledgement}
The author would like to thank his thesis advisor, Professor Nobuhiro Honda and colleagues in my laboratory, Ryota Kotani and Ryoma Saito for thoughtful discussions and constant supports.

\bibliography{reference}

\providecommand{\bysame}{\leavevmode\hbox to3em{\hrulefill}\thinspace}
\providecommand{\MR}{\relax\ifhmode\unskip\space\fi MR }
\providecommand{\MRhref}[2]{%
  \href{http://www.ams.org/mathscinet-getitem?mr=#1}{#2}
}
\providecommand{\href}[2]{#2}
\begin{thebibliography}{10}

\bibitem{Bradlow}
S.~B. Bradlow, \emph{{V}ortices in holomorphic line bundles over closed {K}\"{a}hler manifolds}, Commun. Math. Phys., vol. 135, 1990, pp.~1--17.

\bibitem{Bruzzo}
U.~Bruzzo and B.~G. Otero, \emph{{M}etrics on semistable and numerically effective {H}iggs bundles}, {J}ournal f^^c3^^bcr die reine und angewandte {M}athematik, vol. 2007, 2007, pp.~59--79.

\bibitem{Fujiki}
A.~Fujiki, \emph{Hyper{K}\"{a}hler structure on the moduli space of flat bundles}, Prospects in Complex Geometry, Lecture Notes in Math., vol. 1468, 1991, pp.~1--83.

\bibitem{Griffiths}
P.~A. Griffiths, \emph{The extension problem for compact submanifolds of complex manifolds {I}}, Springer Berlin, Heidelberg, 1965.

\bibitem{Hitchin}
N.~J. Hitchin, \emph{{T}he {S}elf-{D}uality {E}quations on a {R}iemann {S}urface}, Proc. London Math. Soc., vol. s3-55, 1987, pp.~59--126.

\bibitem{Kobayashi}
S.~Kobayashi, \emph{{D}ifferential {G}eometry of {C}omplex {V}ector {B}undles}, Princeton University Press, 1987.

\bibitem{Telemann}
M.~L^^c3^^bcbke and A.~Teleman, \emph{{T}he {K}obayashi-{H}itchin {C}orrespondence}, WORLD SCIENTIFIC, 1995.

\bibitem{Liu}
C.~C. Liu, S.~Rayan, and Y.~Tanaka, \emph{{T}he {K}apustin^^e2^^80^^93{W}itten equations and nonabelian {H}odge theory}, European Journal of Mathematics, vol.~8, 2022, pp.~23--41.

\bibitem{Simpson}
C.~T. Simpson, \emph{{C}onstructing {V}ariations of {H}odge {S}tructure {U}sing {Y}ang-{M}ills {T}heory and {A}pplications to {U}niformization}, J. Amer. Math. Soc., vol.~1, 1988, pp.~867--918.

\bibitem{Simpson2}
\bysame, \emph{{H}iggs bundles and local systems}, Publications Math\'ematiques de l'IH\'ES, vol.~75, 1992, pp.~5--95.

\bibitem{Simpson3}
\bysame, \emph{{M}oduli of representations of the fundamental group of a smooth projective variety {I}}, Publications Math\'ematiques de l'IH\'ES, vol.~79, 1994, pp.~47--129.

\bibitem{Tanaka}
Y.~Tanaka, \emph{{O}n the singular sets of solutions to the {K}apustin^^e2^^80^^93{W}itten equations and the {V}afa^^e2^^80^^93{W}itten ones on compact {K}\"{a}hler surfaces}, Geometriae Dedicata, vol. 199, 2019, pp.~177--187.

\bibitem{Wells}
R.~O. Wells, \emph{{D}ifferential {A}nalysis on {C}omplex {M}anifolds}, Springer New York, NY, 2008.

\end{thebibliography}
\bibliographystyle{amsplain}

\end{document}